\documentclass[nonblindrev]{informs3noheader} %
\usepackage{multibib}
\newcites{app}{References}
\newcolumntype{Z}{>{\centering\arraybackslash}X}
\usepackage{tabularx}
\newcolumntype{R}{>{\raggedleft\arraybackslash}X}
\usepackage{textcomp}
\usepackage{algorithm}
\usepackage{algpseudocode}
\usepackage{url}
\usepackage{bbm}
\usepackage{graphicx}
\usepackage{mathrsfs} 
\usepackage{amsfonts,amssymb} 
\usepackage{threeparttable}
\usepackage{amsmath}
\usepackage{color}
\usepackage{titlesec}
\usepackage{indentfirst}  
\usepackage{enumerate}
\usepackage{bm} 
\usepackage{booktabs} 
\usepackage{makecell}
\usepackage{array}
\usepackage{multirow}
\usepackage{longtable}
\usepackage{subfigure}
\usepackage{balance}
\usepackage{color}
\usepackage[colorlinks, linkcolor=blue,
            citecolor=blue ]{hyperref}
\usepackage{changes}
\usepackage{rotating}
\usepackage{appendix}
\usepackage{adjustbox}
\usepackage{lscape}
\newtheorem{Theorem}{Theorem}

\newtheorem{Lemma}{Lemma}
\newtheorem{Corollary}{Corollary}
\newtheorem{Proposition}{Proposition}
\newtheorem{Remark}{Remark}

\usepackage{natbib}
 \bibpunct[, ]{(}{)}{,}{a}{}{,}%
 \def\bibfont{\small}%

\usepackage{lineno}

\OneAndAHalfSpacedXI 
\TheoremsNumberedThrough     
\EquationsNumberedThrough    

\allowdisplaybreaks[4] 
\usepackage{setspace}

\begin{document}

\RUNTITLE{Traveling Salesman Tardiness}

\TITLE{Traveling Salesman Tardiness}


\ARTICLEAUTHORS{%
\AUTHOR{Haoyue Liu}
\AFF{Division of Logistics and Transportation, Shenzhen International Graduate School, Tsinghua University, Shenzhen 518055, China; \EMAIL{liu-hy22@mails.tsinghua.edu.cn} \URL{}}
\AUTHOR{Sheng Liu\thanks{The work was initiated when the authors visited the National University of Singapore (Department of Analytics and Operations, DAO). The authors thank DAO and the Institute of Operations Research and Analytics for the support and Prof. Melvyn Sim for the stimulating discussions that inspired the technical development of this paper. } }

\AFF{Rotman School of Management, University of Toronto, Toronto, Ontario M5S 3E6, Canada;\\
\EMAIL{sheng.liu@rotman.utoronto.ca}  \URL{}}
}

\ABSTRACT{How fragile is the routing time window of delivery systems against spatial distributional uncertainty?  We study the tardiness risk of Traveling Salesman Problem (TSP) solutions with respect to a service deadline (target) over the routing time. Using the robust satisficing model, we introduce the TSP tardiness index to quantify the target's fragility under distributional uncertainty in customer locations. Assuming there are $m$ potential customer locations from historical samples on a service region $D$ (of area $|D|$), we prove that the TSP tardiness index is $\Theta\left(\frac{n\sqrt{|D|m}}{\tau}\right)$ for $n$ realized locations with respect to the routing time target $\tau$, under non-boundary conditions. This result establishes a new scaling law that extends beyond the existing deterministic and probabilistic TSP bounds. We further extend it to a multi-vehicle case and derive simple partition rules for managing delivery systems. Our numerical experiments using synthetic and real-world routing data validate the value of the TSP tardiness index in characterizing and managing the overtime risk of routing systems. 
}


\maketitle

\vspace{-10pt}
\section{Introduction}
The Traveling Salesman Problem (TSP) and many of its variants/extensions, including the Vehicle Routing Problem (VRP), underpin spatial service operations and supply chains, especially transportation and logistics systems. In particular, running modern parcel delivery operations hinges on designing efficient delivery routes that solve large-scale TSP or VRP instances. These routes are often subject to a delivery time window tied to the promised customer service level. For example, a same-day delivery guarantee can translate into an 8-hour delivery window across the entire route in a typical e-retail setting.\footnote{In this paper, we will use the terms delivery target and time window (more specifically, deadline) interchangeably depending on the context. } However, in practice, violations of the delivery time target are hard to avoid, even when one carefully plans the workload. One reason is that routing time performance is sensitive to variations in spatial demand distributions, and a slight shift in demand locations can have a sizable impact on routing time and the corresponding delay. Quantifying this vulnerability (i.e., target fragility or satisfiability) is practically relevant, as it helps not only generate insights into the logistics system's reliability but also prescribe tactical planning solutions, such as service region dimensioning and fleet sizing. However, the current theoretical understanding of the TSP route-target fragility is limited under distributional uncertainty. 

In this work, we propose a TSP-based risk measure, termed \textit{tardiness index}, that explicitly characterizes the relationship between the target fragility of TSP routes and the distribution of demand locations. Building on the robust satisficing framework, the tardiness index quantifies the fragility of the delivery-time target for a TSP routing system under limited knowledge about demand distributions in a two-dimensional plane. Specifically, it helps assess the risk of missing the delivery target and causing TSP route lateness due to locational uncertainty. We show that the tardiness index can be computed efficiently using a new formulation and derive closed-form scaling laws that demonstrate its dependence on key routing features. These rules can yield simple tactical design principles for real routing systems. We summarize the key contributions below. 

\subsection{Contributions}
The introduced TSP tardiness index addresses the spatial fragility of TSP route targets under distributional shifts in demand (location), measured by the Wasserstein distance to an empirical distribution. By leveraging the structural properties of the worst-case distribution, we derive a second-order cone program for the index computation enabled by quadrature discretization on the two-dimensional plane. This computational strategy achieves orders-of-magnitude improvements in solution time over the alternative cutting-plane method in the literature. Analytically, we prove lower and upper bounds of the TSP tardiness index. The main technical finding is that, under non-boundary cases, the tardiness index scales by $\frac{n\sqrt{|D|m}}{\tau}$, in terms of the number of realized demand points ($n$), the service region area ($|D|$), the number of empirical support points ($m$), and the routing time target ($\tau$). This scaling law is fundamentally different from the classic TSP scaling pattern featured by $\sqrt{n}$ and illustrates the inherent challenge of meeting the TSP route target as demand volume increases. We further extend the scaling law to multi-vehicle cases when the decision maker is concerned with the system's makespan. We then discuss the prescriptive implications of these results based on a region partition application. Empirically, we conduct numerical experiments based on synthetic and real-world routing data, which demonstrate that the TSP tardiness index can serve as a good proxy for out-of-sample overtime risk metrics for delivery routes (including overtime rate and magnitude) and inform more reliable routing system designs than alternative principles. 

\subsection{Related Work}
Stochastic and robust TSPs and VRPs have been extensively studied in the literature, where the source of uncertainty can stem from travel times, demand volume, and locations (see, e.g., \citealt{jaillet1985probabilistic,
laporte1992vehicle,campbell2008probabilistic,
adulyasak2016models, jaillet2016routing, zhang2021robust}, among others). The closest to our problem setting among this literature are \cite{carlsson2013robust} and  \cite{carlsson2018wasserstein}, where the authors assume limited information about the distribution of demand locations in TSPs. Particularly, \cite{carlsson2018wasserstein} consider a Wasserstein distance-based ambiguity set to capture the spatial distributional uncertainty, which we also adopt. However, they focus on algorithm development for districting applications, whereas we characterize the risk performance of the TSP route relative to a prespecified target. To our knowledge, we are the first to provide analytical characterizations of the target violation of TSP routes under spatial distributional uncertainty. 

Methodologically, our work builds on the emerging literature on target-oriented, robust satisficing models for decision making under uncertainty. The robust satisficing framework evaluates decisions based on their ability to meet operational targets, providing a more interpretable, performance-driven approach to risk management \citep{long2023robust}. This framework has been applied to a variety of transportation and logistics settings, such as inventory-routing and logistics network design \citep{che2026robust,hu2024optimal}, shared-mobility system planning \citep{chen2025robust,guo2023data}, on-time and on-demand last-mile delivery \citep{zhao2025analytics, zhao2026service}, and crowdsourcing logistics \citep{cheng2025robust}. 
Different from these studies, which primarily embed robust satisficing rules into specific operational optimization models, we aim to understand the structural properties of the risk measure in the presence of spatial distributional uncertainty. 


\section{TSP Tardiness Index and Its Computation}
This section introduces the tardiness index in the TSP context and presents an efficient formulation for its computation. 
\subsection{Preliminaries}
\label{sec:generalized_rs}
Consider a service region represented by a compact set $D \subset \mathbb{R}^2$ with nonempty interior and Lipschitz boundary. Throughout the paper, $|D|$ denotes the two-dimensional Lebesgue measure, or area, of $D$, and \(\Delta_D:=\mathrm{diam}(D)\) denotes its diameter. For a collection of $n$ demand points in $D$, the underlying spatial demand distribution is represented by a density $f$ belonging to $\mathcal{P}(D):=\left\{f\ge 0:\int_D f(x) dx=1\right\}$. The optimal Euclidean TSP tour length is characterized by the following theorem \citep{beardwood1959shortest}.
\begin{Theorem}[Beardwood-Halton-Hammersley (BHH) Theorem]
    For a set of $n$ points drawn independently from a continuous density $f \in \mathcal{P}(D)$, the length of the optimal TSP tour, denoted by $\mathcal{L}(f)$, satisfies
    \begin{align}
         \lim_{n\to\infty} \frac{\mathcal{L}(f)}{\sqrt{n}} = \beta \int_D \sqrt{f(x)}dx, 
    \end{align}
    \noindent where $\beta$ is the BHH constant.
\end{Theorem}
The BHH theorem is widely used in the routing literature for a first-order approximation of the TSP route length \citep{stroh2022tactical, banerjee2022fleet, carlsson2024provably}. In practice, finite-sample deviations from the BHH formula are expected, primarily due to boundary and shape effects, but are partially absorbed through calibration of the BHH constant $\beta$. The accuracy of such calibrated approximations for planning-level applications has been empirically validated across a range of problem sizes (e.g., up to 7\% errors when $n\geq 50$) and service-region geometries \citep{daganzo1984distance,kwon1995estimating, shen2007incorporating}. In what follows, we use $\mathcal{L}(f)$ to denote the corresponding BHH-based routing length induced by spatial demand density $f$. Without loss of generality, we assume the vehicle travels at a constant speed normalized to one, so tour length and routing time are expressed in the same units. 

The critical input to the TSP route-length computation is the spatial demand density $f$, which cannot be precisely known and must be estimated from data. Given historical demand locations $x_1,\ldots,x_m \in D$, we represent the observed spatial demand pattern by the empirical distribution $\widehat{P}_b=\frac{1}{m}\sum_{r=1}^m \delta_{x_r}$, where $\delta_{x_r}$ denotes the Dirac point mass at $x_r$. Because $\widehat{P}_b$ may deviate from $f$, the routing time estimates and the corresponding delivery planning based on $\widehat{P}_b$ can differ from the ground truth, leaving the system vulnerable to delivery time target violations. To quantify this vulnerability, we follow the satisficing-measure framework of \citet{brown2009satisficing} and adapt the robust satisficing fragility measure of \citet[Definition~2]{long2023robust} to our TSP routing time setting. Specifically, for any target-violation functional $\Psi:\mathcal{P}(D)\to\mathbb{R}$, the \textit{fragility measure} can be defined as
\begin{align}
    \rho(\Psi):=
    \inf\left\{
    k\ge 0:
    \Psi(f)\le k\,W_1(P_f,\widehat{P}_b),\ \forall f\in\mathcal{P}(D)
    \right\},
\end{align}
where $P_f$ denotes the probability measure induced by density $f$ and $W_1(\cdot,\cdot)$ denotes the first-order Wasserstein distance. With a slight abuse of notation, we write $W_1(f,\widehat P_b)$ in place of $W_1(P_f,\widehat P_b)$. Thus, $\rho(\Psi)$ is the smallest proportionality factor that uniformly controls target violation relative to the Wasserstein distance from the empirical distribution; a smaller value indicates lower target fragility and vulnerability to spatial shifts in the demand distribution or misspecification.

\subsection{TSP Tardiness Index}
Suppose the decision-maker specifies a target threshold $\tau$ applied to the TSP routing time.\footnote{We assume the main source of uncertainty is locational in this study. Other time components in a routing system, such as stop-level service time, are assumed to have more reliable estimates, and the specified target already excludes those components. In the real-world numerical study, we control for service-time variability to assess the TSP tardiness index separately. } Let the target-violation functional be $\Psi_\tau(f):=\mathcal{L}(f)-\tau$, and we define the \textit{TSP tardiness index} as
\begin{align}
    \rho_\tau(\mathcal{L}) :=\rho(\Psi_\tau) =
    \inf\left\{k\ge 0:k\ge\sup_{f\in\mathcal{P}(D)}\frac{\bigl(\mathcal{L}(f)-\tau\bigr)^+}{W_1(f,\widehat{P}_b)}
    \right\}.
\end{align}

The index $\rho_\tau(\mathcal{L})$ quantifies the worst-case excess routing time per unit Wasserstein deviation from the empirical distribution. Equivalently, for any density $f$, the excess routing time above $\tau$ is bounded by $\rho_\tau(\mathcal{L})W_1(f,\widehat{P}_b)$, hence a smaller index indicates that the prescribed routing time target is more reliable under spatial demand changes. The following remark formalizes the properties inherited from the robust satisficing framework. These properties support that \(\rho\), and therefore $\rho_\tau$, is a valid target-fragility measure: it is monotone, positively homogeneous, subadditive, and pro-robust when the target is uniformly satisfied.

\begin{Remark}
As formalized in \citet[Theorem~2]{long2023robust}, for any target-violation functional 
\(\Psi:\mathcal P(D)\to\mathbb R\), the functional \(\rho\) satisfies:
\begin{enumerate}[(i)]
    \item \emph{Monotonicity}: If $\Psi_1(f)\le \Psi_2(f)$ for all $f\in\mathcal{P}(D)$, then $\rho(\Psi_1)\le \rho(\Psi_2)$.
    \item \emph{Positive homogeneity}: For any $a\ge 0$, $\rho(a\Psi)=a\,\rho(\Psi)$.
    \item \emph{Subadditivity}: For any $\Psi_1,\Psi_2:\mathcal{P}(D)\to\mathbb{R}$, $\rho(\Psi_1+\Psi_2)\le \rho(\Psi_1)+\rho(\Psi_2)$.
    \item \emph{Prorobustness}: If $\Psi(f)\le 0$ for all $f\in\mathcal{P}(D)$, then $\rho(\Psi)=0$.
\end{enumerate}
The antifragility property in \citet[Theorem~2]{long2023robust} requires an extension to the nominal measure and is therefore not used in our density-based BHH analysis.
\end{Remark}
To facilitate the analysis and computation of $\rho_\tau(\mathcal{L})$, it is convenient to introduce the fixed-$k$ value function:
\begin{align}
    V_D(k):=\sup_{f\in \mathcal{P}(D)}\Bigl\{\beta\sqrt{n}\int_D \sqrt{f(x)}\,dx-kW_1(f,\widehat{P}_b)\Bigr\}.
\end{align}
It is straightforward to verify that $V_D(k)$ is monotonically nonincreasing in $k$. Consequently, the tardiness index can be equivalently characterized as $\rho_\tau(\mathcal{L})=\inf\{k\ge0: V_D(k)\le \tau\}$. This equivalence implies that once $V_D(k)$ can be efficiently evaluated, the tardiness index $\rho_\tau(\mathcal{L})$ can be computed via a one-dimensional bisection search over $k$. To tractably evaluate the Wasserstein distance between a continuous density $f$ and the discrete empirical measure $\widehat{P}_b$, we exploit the duality of semidiscrete optimal transport \citep{villani2021topics}. By introducing the bounded parameter set $\Lambda := \{ \lambda \in \mathbb{R}^m : e^\top \lambda = 0, \ \lambda_i \le \Delta_D:=\operatorname{diam}(D)\}$, the Wasserstein distance admits an exact finite-dimensional dual representation: 
\begin{align}
    W_1(f, \widehat{P}_b) = \max_{\lambda \in \Lambda} \int_D a_\lambda(x)f(x)dx, \ a_\lambda(x) := \min_{r} \{ \|x - x_r\| - \lambda_r \}.\label{eq:wass}
\end{align} 
\noindent Substituting this dual form into $V_D(k)$ gives:
\begin{align}
    V_D(k) = \sup_{f\in \mathcal{P}(D)}\inf_{\lambda \in \Lambda}\int_D \left(\beta\sqrt{n}\sqrt{f(x)} - k a_\lambda(x)f(x)\right) dx.
\end{align}

While this structural form yields a concave-convex geometry, the outer maximization remains infinite-dimensional over the Banach space $L^1(D)$. Proposition~\ref{prop:single_zone_dual} shows that the fixed-$k$ problem admits an exact finite-dimensional reformulation.

\begin{Proposition}
\label{prop:single_zone_dual}
For every fixed $k\ge 0$, the fixed-$k$ value function admits the finite-dimensional representation:
\begin{equation}
\label{eq:Vk_dual}
V_D(k)
=
\min_{\substack{\lambda\in\Lambda,\ \nu\in\mathbb R\\
ka_\lambda(x)+\nu>0,\ \forall x\in D}}
\left\{
\nu+\frac{\beta^2 n}{4}\int_D \frac{dx}{ka_\lambda(x)+\nu}
\right\}.
\end{equation}
Furthermore, letting $(\lambda^\star,\nu^\star)$ be a minimizer of \eqref{eq:Vk_dual}, the corresponding worst-case density is uniquely given by:
$$f_k^\star(x)=\frac{\beta^2n}{4\bigl(ka_{\lambda^\star}(x)+\nu^\star\bigr)^2},\quad x\in D.$$
\end{Proposition}

\textit{Proof.} The proof is organized in three steps.

\noindent\textit{Step 1: Weak duality.}
For any fixed $\lambda\in\Lambda$, define
$$H_k(\lambda):=\sup_{f\in\mathcal P(D)}\int_D\Bigl(\beta\sqrt n\sqrt{f(x)}-ka_\lambda(x)f(x)\Bigr)\,dx.$$ 
By weak duality, $V_D(k)\le H_k(\lambda)$ for every $\lambda\in\Lambda$, and hence
$V_D(k)\le \inf_{\lambda\in\Lambda}H_k(\lambda)$. Note that the set $\Lambda$ is compact, since $e^\top\lambda=0$ and $\lambda_i\le \Delta_D$ imply $\lambda_i\ge -(m-1)\Delta_D$ for each $i$.

\medskip
\noindent\textit{Step 2: Analytical evaluation of $H_k(\lambda)$.}
Fix $\lambda\in\Lambda$. Dualizing the constraint $\int_D f(x)\,dx=1$ with multiplier $\nu\in\mathbb R$ gives
$$H_k(\lambda)\le\inf_{\nu\in\mathbb R}\sup_{f\in L_+^1(D)}\left\{
\nu+\int_D\Bigl[\beta\sqrt n\sqrt{f(x)}-\bigl(ka_\lambda(x)+\nu\bigr)f(x)\Bigr]dx
\right\}.$$ 
If $ka_\lambda(x)+\nu\le 0$ at some $x$, the inner supremum is $+\infty$ by increasing $f$ on a neighborhood of that point. Hence finiteness requires $ka_\lambda(x)+\nu>0,\ \forall x\in D$. Under this condition, the maximization separates pointwise. For $y:=ka_\lambda(x)+\nu>0$ and $u:=f(x)\ge 0$, the scalar problem
$$\sup_{u\ge 0}\{\beta\sqrt n\sqrt u-yu\}$$
is strictly concave and has the unique maximizer $u^\star=\beta^2n/(4y^2)$, with optimal value $\beta^2n/(4y)$. Therefore,
$$H_k(\lambda)
\le
\inf_{\substack{\nu\in\mathbb R\\ ka_\lambda(x)+\nu>0,\ \forall x\in D}}
J_\lambda(\nu),
\qquad
J_\lambda(\nu):=\nu+\frac{\beta^2n}{4}\int_D\frac{dx}{ka_\lambda(x)+\nu},$$
and the associated pointwise maximizer is
$$f_{\lambda,\nu}^\star(x)=\frac{\beta^2n}{4\bigl(ka_\lambda(x)+\nu\bigr)^2}.$$ 
By strict convexity and boundary behavior, \(J_\lambda\) admits a unique minimizer $\nu_\lambda^\star$. Indeed, \(J_\lambda''(\nu)>0\), \(J_\lambda'(\nu)\to1\) as \(\nu\to\infty\), and \(J_\lambda'(\nu)\to-\infty\) as \(\nu\downarrow -k\min_x a_\lambda(x)\), where the last limit follows from the logarithmic divergence of \(\int_D(ka_\lambda(x)+\nu)^{-2}dx\) near a minimizer of \(a_\lambda\). The first-order condition $J_\lambda'(\nu_\lambda^\star)=0$ ensures exactly that $\int_D f^\star_{\lambda,\nu^\star_\lambda}(x)dx=1$. Thus, $f^\star_{\lambda,\nu^\star_\lambda} \in \mathcal P(D)$ and the upper bound is attained, yielding $H_k(\lambda) = J_\lambda(\nu_\lambda^\star)$.

\medskip
\noindent
\textit{Step 3: Attainment and exactness.}
By standard continuity arguments of parametric optimization and the compactness of $\Lambda$, the function $\lambda\mapsto H_k(\lambda)$ is continuous and attains its minimum at some $\lambda^\star\in\Lambda$. Let $\nu^\star:=\nu_{\lambda^\star}^\star$ and define $f_k^\star := f^\star_{\lambda^\star,\nu^\star}$. To establish exactness, assume $k>0$ (the case $k=0$ is trivial). For
\[
F(\lambda):=\int_D a_\lambda(x)f_k^\star(x)\,dx,
\]
the maximizer in the definition of $H_k(\lambda^\star)$ is unique, namely $f_k^\star$, by the strict concavity of the inner problem in $f$. Moreover, $H_k$ is convex on $\Lambda$ as the pointwise supremum of convex functions. Hence Danskin's theorem gives
\[
\partial H_k(\lambda^\star)=k\,\partial (-F)(\lambda^\star).
\]
Since $\lambda^\star$ minimizes $H_k$ over $\Lambda$, $0\in \partial H_k(\lambda^\star)+N_\Lambda(\lambda^\star)$, where $N_\Lambda(\lambda^\star)$ denotes the normal cone of $\Lambda$ at $\lambda^\star$. Therefore, we have $0\in k\,\partial (-F)(\lambda^\star)+N_\Lambda(\lambda^\star)$. Because $k>0$ and $N_\Lambda(\lambda^\star)$ is a cone, this is equivalent to $0\in \partial (-F)(\lambda^\star)+N_\Lambda(\lambda^\star)$, which is exactly the first-order optimality condition minimizing the convex function $-F$ over $\Lambda$, or equivalently, for maximizing $F$ over $\Lambda$. By the semidiscrete Kantorovich dual representation, $W_1(f_k^\star,\widehat P_b)=\max_{\lambda\in\Lambda}F(\lambda)=F(\lambda^\star)$. Consequently,
\[
V_D(k)\le H_k(\lambda^\star)
=\beta\sqrt n\int_D\sqrt{f_k^\star(x)}\,dx-k W_1(f_k^\star,\widehat P_b)
\le V_D(k),
\]
where the last inequality follows from the primal feasibility of $f_k^\star$. This establishes the representation. Finally, the strict concavity of the primal objective on $\mathcal P(D)$ implies that $f_k^\star$ is the unique worst-case density.\halmos

We note that the reformulation in Proposition~\ref{prop:single_zone_dual} can be viewed as a penalized counterpart of Wasserstein distributionally robust TSP studied by \citet{carlsson2018wasserstein}: instead of fixing a Wasserstein radius and maximizing the routing time over all distributions within that radius, we penalize the Wasserstein deviation by $k$ and arrive at a similar worst-case structure but with a different primal-dual argument. 
Although Proposition~\ref{prop:single_zone_dual} reduces the fixed-$k$ evaluation to finite dimensions in $(\lambda,\nu)$, the objective still contains an intractable continuous integral over $D$. To achieve computational tractability, we approximate this spatial integration using numerical quadrature. Specifically, this discretization not only makes the objective evaluable but also enables us to transform the problem into a second-order cone program (SOCP), all while preserving the exact semidiscrete optimal transport structure. Let $\{(z_c,w_c)\}_{c=1}^C$ be a quadrature rule on $D$ with strictly positive weights, and define the distance matrix elements $d_{ci}:=\|z_c-x_i\|$ for $c=1,\ldots,C$ and $i=1,\ldots,m$. To guarantee strict positivity of the denominator and ensure numerical stability, we introduce a tolerance $\varepsilon>0$. The discretized problem then admits the following SOCP reformulation.

\begin{Proposition}
\label{prop:socp_fixed_k}
For any fixed $k>0$, quadrature rule $\{(z_c,w_c)\}_{c=1}^C$, and tolerance $\varepsilon>0$, evaluating the discretized fixed-$k$ value function is equivalent to the following SOCP:
\begin{subequations}
\begin{align}
\min_{\lambda,\nu,s,u,t,q}\quad 
& \nu+\frac{\beta^2n}{4}\sum_{c=1}^C w_c q_c
\label{eq:socp_obj}
\\
\text{s.t.}\quad
& u_c\le d_{ci}-\lambda_i,
\qquad c=1,\ldots,C,\ \ i=1,\ldots,m,
\label{eq:socp_u}
\\
& t_c=ku_c+\nu,
\qquad c=1,\ldots,C,
\label{eq:socp_t}
\\
& \left(\frac{q_c}{2},\,t_c,\,1\right)\in\mathcal K_r,
\qquad c=1,\ldots,C,
\label{eq:socp_rot}
\\
& s\ge \lambda_i,
\qquad i=1,\ldots,m,
\label{eq:socp_s}
\\
& \nu\ge ks+\varepsilon,
\label{eq:socp_pos}
\\
& e^\top\lambda=0,\qquad \lambda_i\le \Delta_D,
\qquad i=1,\ldots,m,
\label{eq:socp_lambda}
\end{align}
\end{subequations}
where $\mathcal K_r:=\{(x,y,z)\in\mathbb R^3:x\ge 0,\ y\ge 0,\ 2xy\ge z^2\}$ is the rotated second-order cone.
\end{Proposition}

\textit{Proof. }
For any feasible solution, the objective minimizes $q_c$. The rotated cone constraints \eqref{eq:socp_rot} enforce $q_c \ge 1/t_c$, ensuring $q_c = 1/t_c$ at optimality. Because the map $u \mapsto 1/(ku+\nu)$ is strictly decreasing for $ku+\nu > 0$, the minimization pushes $u_c$ to its upper bound. Thus, constraints \eqref{eq:socp_u} bind at $u_c = \min_i \{d_{ci}-\lambda_i\} = a_\lambda(z_c)$. Substituting this into \eqref{eq:socp_t} gives $t_c = ka_\lambda(z_c)+\nu$, which yields exactly the quadrature-discretized objective. To enforce the strict positivity condition $ka_\lambda(x)+\nu \ge \varepsilon$ globally over $D$, we exploit the identity $\min_{x\in D} a_\lambda(x) = -\max_i \lambda_i$. Constraint \eqref{eq:socp_s} defines $s \ge \max_i \lambda_i$, allowing \eqref{eq:socp_pos} to enforce $\nu - k \max_i \lambda_i \ge \varepsilon$. This guarantees the required bounded positivity for all $x \in D$, confirming the exact equivalence. \halmos

Propositions~\ref{prop:single_zone_dual} and \ref{prop:socp_fixed_k} offer the computational foundation for the TSP tardiness index. They show that the infinite-dimensional fixed-\(k\) problem can be reduced to a finite-dimensional reformulation and then transformed into an SOCP through spatial quadrature. Consequently, \(\rho_\tau(L)\) can be efficiently computed by solving the SOCP within a one-dimensional bisection search over \(k\). It is also worth noting that our SOCP reformulation can be applied to solve the distributionally robust formulation in \citet{carlsson2018wasserstein}, yielding significant computational improvements over the cutting-plane algorithm proposed there. Detailed comparisons are presented in~\ref{EC:computation}. The next section derives analytical bounds and establishes scaling laws for the TSP tardiness index that shed light on the target satisfiability of delivery systems.

\section{Scaling of TSP Tardiness: Bounds and Analysis}
\label{sec:scaling}
To study how the TSP tardiness index changes with the target level and the geometry of the service region, we first reduce its evaluation to a one-dimensional variational problem. The key step is to separate the Wasserstein deviation level from the BHH routing-time expression. For $0<t\le \Delta_D:=\operatorname{diam}(D)$, define the radius-constrained envelope function 
\begin{equation}
\label{eq:radius_envelope}
\mathcal E_D(t):=\sup\left\{\int_D \sqrt{f(x)}\,dx: f\in\mathcal P(D),\ W_1(f,\widehat P_b)\le t\right\}.
\end{equation}
Specifically,  $\mathcal E_D(t)$ is a scalar-valued function of the Wasserstein radius \(t\), capturing the largest BHH routing-time factor attainable within distance \(t\) from the empirical distribution. Once this envelope function is characterized, both the fixed-\(k\) value function $V_D(k)$ and the tardiness index $\rho_\tau(\mathcal{L})$ admit one-dimensional variational representations, as shown in the following lemma.

\begin{Lemma}
\label{lem:Vk_kappa_radius}
For every $k\ge 0$,
\begin{equation}
\label{eq:Vk_radius}
V_D(k)=\sup_{0<t\le \Delta_D}\left\{\beta\sqrt n\,\mathcal E_D(t)-kt\right\}.
\end{equation}
Consequently, for every target $\tau>0$,
\begin{equation}
\label{eq:kappa_radius}
\rho_\tau(\mathcal{L})=\sup_{0<t\le \Delta_D}\frac{\bigl(\beta\sqrt n\,\mathcal E_D(t)-\tau\bigr)^+}{t}.
\end{equation}
\end{Lemma}

\textit{Proof.}
Let $k\ge 0$. For any $f\in\mathcal P(D)$, define $t_f:=W_1(f,\widehat P_b)\in(0,\Delta_D]$. By the definition of $\mathcal E_D$, we have $\int_D \sqrt{f(x)} dx\le \mathcal E_D(t_f)$, and hence
\[
\beta\sqrt n\int_D\sqrt{f(x)}\,dx-kW_1(f,\widehat P_b)
\le \beta\sqrt n\,\mathcal E_D(t_f)-kt_f
\le \sup_{0<t\le \Delta_D}\{\beta\sqrt n\,\mathcal E_D(t)-kt\}.
\]
Taking the supremum over $f\in\mathcal P(D)$ gives
\[
V_D(k)\le \sup_{0<t\le \Delta_D}\{\beta\sqrt n\,\mathcal E_D(t)-kt\}.
\]
Conversely, for any $t\in(0,\Delta_D]$ and $\varepsilon>0$, the definition of $\mathcal E_D(t)$ yields some $f_{t,\varepsilon}\in\mathcal P(D)$ such that
\[
W_1(f_{t,\varepsilon},\widehat P_b)\le t,
\qquad
\int_D\sqrt{f_{t,\varepsilon}(x)}\,dx\ge \mathcal E_D(t)-\varepsilon.
\]
Therefore,
\[
V_D(k)\ge\beta\sqrt n\int_D\sqrt{f_{t,\varepsilon}(x)}\,dx-kW_1(f_{t,\varepsilon},\widehat P_b)\ge\beta\sqrt n\bigl(\mathcal E_D(t)-\varepsilon\bigr)-kt.
\]
Letting $\varepsilon\downarrow 0$ and then taking the supremum over $t$ establishes \eqref{eq:Vk_radius}. The identity \eqref{eq:kappa_radius} then follows directly from \eqref{eq:Vk_radius}, because the condition $V_D(k)\le \tau$ is equivalent to
\[
k\ge \frac{\beta\sqrt n\,\mathcal E_D(t)-\tau}{t}
\qquad \forall\, t\in(0,\Delta_D],
\]
or, equivalently,
\[
k\ge \frac{\bigl(\beta\sqrt n\,\mathcal E_D(t)-\tau\bigr)^+}{t}
\qquad \forall\, t\in(0,\Delta_D].
\]
Taking the smallest admissible $k$ yields \eqref{eq:kappa_radius}.  \halmos

Lemma~\ref{lem:Vk_kappa_radius} is the key reduction for the scaling analysis. It shows that the TSP tardiness index is fully determined by the growth of the radius-constrained envelope \(\mathcal E_D(t)\). Hence, the rest of the analysis can focus on bounding \(\mathcal E_D(t)\), and these bounds can translate directly into upper and lower bounds for \(\rho_\tau(\mathcal{L})\).

\subsection{Upper Bound} 

We first derive a distribution-free upper bound on $\mathcal{E}_D(t)$ using only the geometry of the empirical support. Let $X := \{x_1, \ldots, x_m\}$ denote the support of the empirical measure $\widehat P_b$, and define the distance-to-support function $d_X(x) := \min_{1\le r\le m}\|x-x_r\|$ for $x\in D$. The next proposition links the BHH routing-time factor to the Wasserstein distance from the empirical distribution. 

\begin{Proposition}
\label{prop:radius_upper}For every density $f\in\mathcal P(D)$,
\begin{equation}
\label{eq:pointwise_upper}
\int_D \sqrt{f(x)} dx
\le \sqrt{I_X\,W_1(f,\widehat P_b)},
\qquad  \text{where} \quad
I_X:=\int_D\frac{dx}{d_X(x)}.
\end{equation}
Moreover, the geometric factor $I_X$ satisfies\begin{equation}
\label{eq:IX_upper}
I_X\le 2\sqrt{\pi |D|m}.
\end{equation}
Consequently, the radius-constrained envelope is bounded by
\begin{equation}
\label{eq:radius_upper_envelope}
\mathcal E_D(t)\le\sqrt{2\sqrt{\pi |D|m}t}=\sqrt{2}(\pi |D|m)^{1/4}\sqrt t,
\qquad 0<t\le \Delta_D.
\end{equation}
\end{Proposition}

\textit{Proof.}
The proof proceeds in three steps: we first evaluate the Wasserstein distance using the distance-to-support function, combine the ensuing first-moment bound with the Cauchy–Schwarz inequality, and finally bound the geometric integral via a Voronoi-cell analysis.

\medskip
\noindent\textit{Step 1: A transport bound obtained from the distance-to-support function.}
Observe that $d_X$ is $1$-Lipschitz on $D$, and $d_X(x_r)=0$ for every empirical atom $x_r \in X$. Therefore, the Kantorovich–Rubinstein dual representation of the Wasserstein distance implies
\[
\int_D d_X(x)f(x)\,dx-\int_D d_X(x)\,d\widehat P_b(x)
\le W_1(f,\widehat P_b).
\]
Since $\widehat P_b$ is supported on $X$, the second integral vanishes:
\[
\int_D d_X(x)\,d\widehat P_b(x)=\frac1m\sum_{r=1}^m d_X(x_r)=0.
\]
Thus, we obtain the first-moment estimate
\begin{equation}
\label{eq:dX_transport}
\int_D d_X(x)f(x)\,dx\le W_1(f,\widehat P_b).
\end{equation}

\medskip
\noindent\textit{Step 2: Deriving the routing time bound from the first-moment estimate.}
Since $X$ is a finite set, it is of Lebesgue measure zero. Hence, almost everywhere on $D$ (specifically, on $D\setminus X$), we can rewrite the integrand as
\[
\sqrt{f(x)}=\sqrt{d_X(x)f(x)}\,d_X(x)^{-1/2}.
\]
Combining this with \eqref{eq:dX_transport} immediately yields
\[
\int_D \sqrt{f(x)}\,dx
=
\int_{D\setminus X} \sqrt{d_X(x)f(x)}\,\frac{dx}{\sqrt{d_X(x)}}.
\]
An application of the Cauchy–Schwarz inequality, coupled with the fact that $I_X < \infty$ (as will be shown in Step 3), gives
\[
\left(\int_D \sqrt{f(x)}\,dx\right)^2
\le
\left(\int_D d_X(x)f(x)\,dx\right)
\left(\int_D \frac{dx}{d_X(x)}\right).
\]
Combining this with \eqref{eq:dX_transport} yields
\[
\left(\int_D \sqrt{f(x)}\,dx\right)^2
\le I_X\,W_1(f,\widehat P_b),
\]
which establishes \eqref{eq:pointwise_upper}.

\medskip
\noindent\textit{Step 3: Bounding the geometric factor $I_X$.}
Let $\{C_r\}_{r=1}^m$ be a measurable Voronoi partition of $D$ generated by $X$. By the definition of this partition, the minimum distance simplifies locally to $d_X(x) = \|x-x_r\|$ for almost every $x\in C_r$. This allows us to decompose the integral as
\[
I_X=\sum_{r=1}^m\int_{C_r}\frac{dx}{\|x-x_r\|}.
\]
For a fixed area $a_r$, the kernel $u \mapsto \|u\|^{-1}$ is radially symmetric and strictly decreasing. By symmetric rearrangement, the integral over the cell $C_r$ is maximized when the domain is replaced by a centered disk of equivalent area. Consequently,
\[
\int_{C_r}\frac{dx}{\|x-x_r\|}\le
\int_{B(0,\sqrt{a_r/\pi})}\frac{du}{\|u\|}=2\sqrt{\pi a_r}.
\]
Aggregating this bound across all $m$ cells and invoking the Cauchy–Schwarz inequality, we obtain
\[
I_X\le2\sqrt\pi\sum_{r=1}^m \sqrt{a_r}\le2\sqrt\pi\,\sqrt{m\sum_{r=1}^m a_r}
=2\sqrt{\pi |D|m},
\]
which establishes \eqref{eq:IX_upper}, where $|D|$ is the area of the zone. Finally, taking the supremum of \eqref{eq:pointwise_upper} over the Wasserstein ball $\{f \in \mathcal{P}(D) : W_1(f, \widehat P_b) \le t\}$ yields the envelope bound \eqref{eq:radius_upper_envelope}. \halmos

Proposition~\ref{prop:radius_upper} links the envelope function with the basic geometric features. It implies that, regardless of the empirical configuration, the BHH routing-time factor cannot grow faster than order \(|D|^{1/4}m^{1/4}\sqrt t\) within a Wasserstein radius \(t\). Combining this geometric bound with Lemma~\ref{lem:Vk_kappa_radius} gives a closed-form upper bound for the TSP tardiness index.

\begin{Corollary}
\label{cor:kappa_upper}
Letting $A_{\mathrm{up}}:=\beta\sqrt{2n}(\pi |D|m)^{1/4}$, the tardiness index satisfies
\begin{equation}
\label{eq:kappa_upper_scalar}
\rho_\tau(\mathcal{L})
\le
\sup_{0<t\le \Delta_D}
\frac{\bigl(A_{\mathrm{up}}\sqrt t-\tau\bigr)^+}{t}.
\end{equation}
Moreover, the supremum evaluates to the piecewise function
\begin{equation}
\label{eq:kappa_upper_piecewise}
\sup_{0<t\le \Delta_D}
\frac{\bigl(A_{\mathrm{up}}\sqrt t-\tau\bigr)^+}{t}
=
\begin{cases}
\dfrac{A_{\mathrm{up}}^2}{4\tau}, & 0<\tau\le \dfrac{A_{\mathrm{up}}\sqrt{\Delta_D}}{2},\\
\dfrac{A_{\mathrm{up}}\sqrt{\Delta_D}-\tau}{\Delta_D}, & \dfrac{A_{\mathrm{up}}\sqrt{\Delta_D}}{2}<\tau<A_{\mathrm{up}}\sqrt{\Delta_D},\\
0, & \tau\ge A_{\mathrm{up}}\sqrt{\Delta_D}.
\end{cases}
\end{equation}
In particular, whenever $0<\tau\le A_{\mathrm{up}}\sqrt{\Delta_D}/2$,
\begin{equation}
\label{eq:kappa_upper_interior}
\rho_\tau(\mathcal{L})
\le
\frac{A_{\mathrm{up}}^2}{4\tau}
=
\frac{\beta^2 n}{2\tau}\sqrt{\pi |D|m}.
\end{equation}
\end{Corollary}

\textit{Proof.}
The inequality \eqref{eq:kappa_upper_scalar} follows directly by substituting the envelope bound \eqref{eq:radius_upper_envelope} into the representation of $\rho_\tau(\mathcal{L})$ given in \eqref{eq:kappa_radius}. To evaluate this scalar supremum, observe that the positive part vanishes whenever $A_{\mathrm{up}}\sqrt t\le \tau$. Thus, it suffices to analyze the active region $t > (\tau/A_{\mathrm{up}})^2$, over which the objective function reduces to
\[
h(t):=\frac{A_{\mathrm{up}}\sqrt t-\tau}{t}=\frac{A_{\mathrm{up}}}{\sqrt t}-\frac{\tau}{t}.
\]
Differentiating $h$ with respect to $t$ yields
\[
h'(t)=\frac{1}{t^2}\left(\tau-\frac{A_{\mathrm{up}}}{2}\sqrt t\right),
\]
which admits a unique unconstrained maximizer at $t^*=\frac{4\tau^2}{A_{\mathrm{up}}^2}$. To determine the constrained supremum on $(0, \Delta_D]$, we consider three regimes for the target $\tau$: i) If $t^* \le \Delta_D$, which occurs if and only if $0 < \tau \le A_{\mathrm{up}}\sqrt{\Delta_D}/2$, the maximizer $t^*$ is feasible. Evaluating the objective at this interior peak yields $h(t^*) = A_{\mathrm{up}}^2 / (4\tau)$. ii) If $t^* > \Delta_D$, corresponding to $A_{\mathrm{up}}\sqrt{\Delta_D}/2 < \tau < A_{\mathrm{up}}\sqrt{\Delta_D}$, we have $h'(t) > 0$ strictly throughout the feasible active interval. Consequently, the maximum is attained at the right boundary $t = \Delta_D$, yielding $h(\Delta_D) = (A_{\mathrm{up}}\sqrt{\Delta_D} - \tau) / \Delta_D$. iii) If $\tau \ge A_{\mathrm{up}}\sqrt{\Delta_D}$, the condition $A_{\mathrm{up}}\sqrt t - \tau \le 0$ holds for all $t \in (0, \Delta_D]$. The active region is empty, and the supremum is trivially $0$, thereby proving the corollary. \halmos

Corollary~\ref{cor:kappa_upper} shows that the upper bound on the TSP tardiness index exhibits a three-regime structure. In the tight-target regime, the maximizing Wasserstein radius is interior and the bound reduces to $\rho_\tau(L) \le \frac{\beta^2 n}{2\tau}\sqrt{\pi |D|m}$. 
As expected, larger demand volume, denser potential demand points, or a larger service area all amplify the potential target violation induced by spatial demand distributional shifts, whereas a more relaxed target reduces this sensitivity. Notably, the \(\sqrt m\) term reflects the geometry of a unit-Wasserstein sensitivity measure rather than a degradation from using more data: a denser empirical support creates more local perturbation directions under a fixed Wasserstein budget. Moreover, in contrast to the standard BHH scaling with $\sqrt{n}$, the TSP tardiness index could scale linearly in $n$, reflecting that target violations become increasingly difficult to avoid as demand volume grows, and the economies of scale in TSP do not immediately translate to routing reliability.  When the target becomes less restrictive, the maximizing Wasserstein radius reaches the boundary \(\Delta_D\), and the bound decreases linearly in \(\tau\). Once the target is sufficiently large, the upper bound becomes zero, indicating that the target is stable under all possible perturbations covered by the global envelope.
\subsection{Lower Bound} 
The upper-bound analysis identifies a natural growth rate for the radius-constrained envelope $\mathcal{E}_D(t)$. To obtain a meaningful lower bound, we require a regularity condition that ensures the empirical support contains sufficient local interior neighborhoods for feasible perturbations, which is formalized as follows.

\begin{assumption}[Interior support dispersion]
\label{ass:interior_dispersion}
There exist constants $\eta\in(0,1]$ and $\zeta>0$, both independent of $m$ \textcolor{black}{and of the scale of $D$}, and an index set $J_m\subseteq\{1,\ldots,m\}$ of cardinality $|J_m|\ge \eta m$, such that the balls 
\[
\textcolor{black}{\bigl\{B(x_j,\zeta |D|^{1/2}m^{-1/2}):j\in J_m\bigr\}}
\]
are pairwise disjoint and contained in $D$.
\end{assumption}

Assumption~\ref{ass:interior_dispersion} only requires that a positive fraction of the empirical atoms remain separated at the natural planar scale $|D|^{1/2}m^{-1/2}$ and stay away from the boundary. In particular, it rules out severe clustering and boundary concentration, but does not impose any lattice structure or global regularity. Under this condition, we develop a local lower bound on $\mathcal E_D(t)$ by concentrating probability mass around any generic subset of well-separated atoms in Proposition \ref{prop:radius_lower_general}.

\begin{Proposition}
\label{prop:radius_lower_general}
Let $\mathcal{J} = \{j_1,\ldots,j_q\}\subseteq\{1,\ldots,m\}$ be a collection of distinct indices, and let $\{r_s\}_{s=1}^q$ be positive radii such that the balls $B(x_{j_s},r_s)$ are pairwise disjoint and contained in $D$. Defining $S_q := \sum_{s=1}^q r_s$, one has 
\begin{equation}
\label{eq:radius_lower_general}
\mathcal E_D(t)\ge \frac{4}{3}\sqrt{\pi S_q\,t},\quad \forall t \in (0, S_q / (2m)].
\end{equation}
Consequently, if Assumption \ref{ass:interior_dispersion} holds, defining $\overline{t}_{\mathrm{low}} := \frac{\eta\zeta\sqrt{|D|}}{2\sqrt m}$, the radius-constrained envelope satisfies$$\label{eq:radius_lower_envelope}
\mathcal E_D(t) \ge \frac{4}{3}\sqrt{\pi\eta\zeta}|D|^{1/4}m^{1/4}\sqrt t, \quad \forall t \in (0, \overline{t}_{\mathrm{low}}].$$

\end{Proposition}

\textit{Proof.}
For any given target distance $t\in(0,S_q/(2m)]$, we establish the lower bound by explicitly constructing a feasible composite density characterized by isolated radial peaks in four steps.

\medskip
\noindent\textit{Step 1: A unit-mass radial density on each selected ball.}
 For each index $s=1,\ldots,q$, we define the normalized radial density as
\[
h_s(x):=\frac{1}{2\pi r_s\|x-x_{j_s}\|}\,\mathbf 1_{\{\|x-x_{j_s}\|\le r_s\}}.
\]
A direct polar-coordinate calculation yields
\[
\int_D h_s(x)\,dx=1,\qquad
\int_D \sqrt{h_s(x)}\,dx=\frac{2\sqrt{2\pi}}{3}\,r_s,\qquad
\int_D \|x-x_{j_s}\|h_s(x)\,dx=\frac{r_s}{2}.
\]

\medskip
\noindent\textit{Step 2: Construction of a feasible density.}
 Let $\delta\in(0,t)$ and set $\alpha_s:=2(t-\delta)/S_q$ for $s=1,\ldots,q$. Since $t\le S_q/(2m)$, one has $0<\alpha_s<2t/S_q\le 1/m$, so $m^{-1}-\alpha_s\ge 0$. Let $\varepsilon\in(0,\delta]$ be sufficiently small, and define
\[
g_{\ell,\varepsilon}(x):=\frac{\mathbf 1_{D\cap B(x_\ell,\varepsilon)}(x)}{|D\cap B(x_\ell,\varepsilon)|},
\qquad \ell=1,\ldots,m.
\]
Because the sample set is finite and $D$ has nonempty interior and Lipschitz boundary, $\varepsilon$ can be chosen so that $|D\cap B(x_\ell,\varepsilon)|>0$ for every $\ell$. Let
\[
f_{\delta,\varepsilon}(x):=
\sum_{s=1}^q \alpha_s h_s(x)
+\sum_{s=1}^q\Bigl(\frac1m-\alpha_s\Bigr)g_{j_s,\varepsilon}(x)
+\sum_{\ell\notin\{j_1,\ldots,j_q\}}\frac1m\,g_{\ell,\varepsilon}(x).
\]
Since each $h_s$ and $g_{\ell,\varepsilon}$ integrates to one, it follows that $\int_D f_{\delta,\varepsilon}(x)\,dx=1$, and hence $f_{\delta,\varepsilon}\in\mathcal P(D)$.

\medskip
\noindent\textit{Step 3: Lower-bounding the objective.}
 Because the balls $B(x_{j_s},r_s)$ are pairwise disjoint and $f_{\delta,\varepsilon}(x)\ge \alpha_s h_s(x)$ on each such ball,
\[
\int_D \sqrt{f_{\delta,\varepsilon}(x)}\,dx
\ge
\sum_{s=1}^q \int_{B(x_{j_s},r_s)}\sqrt{\alpha_s h_s(x)}\,dx
=
\sum_{s=1}^q \sqrt{\alpha_s}\int_D \sqrt{h_s(x)}\,dx.
\]
Using the identities from Step 1 and $\sum_{s=1}^q r_s=S_q$, we obtain
\[
\int_D \sqrt{f_{\delta,\varepsilon}(x)}\,dx
\ge
\frac{2\sqrt{2\pi}}{3}\sum_{s=1}^q r_s\sqrt{\frac{2(t-\delta)}{S_q}}
=
\frac{4}{3}\sqrt{\pi S_q(t-\delta)}.
\]

\medskip
\noindent\textit{Step 4: Verification of Wasserstein feasibility.} It remains to show that the constructed density lies within the prescribed Wasserstein radius.
Consider the transport plan that sends the mass $\alpha_s h_s(x)\,dx$ and $(m^{-1}-\alpha_s)g_{j_s,\varepsilon}(x)\,dx$ to $x_{j_s}$ for each $s$, and sends the mass $m^{-1}g_{\ell,\varepsilon}(x)\,dx$ to $x_\ell$ for each $\ell\notin\{j_1,\ldots,j_q\}$. This defines a feasible coupling between $f_{\delta,\varepsilon}(x)\,dx$ and $\widehat P_b$, and its cost is at most
\[
\sum_{s=1}^q \alpha_s\int_D \|x-x_{j_s}\|h_s(x)\,dx+\varepsilon
=
\frac12\sum_{s=1}^q \alpha_s r_s+\varepsilon
=
t-\delta+\varepsilon
\le t.
\]
Hence $W_1(f_{\delta,\varepsilon},\widehat P_b)\le t$, so $f_{\delta,\varepsilon}$ is feasible for $\mathcal E_D(t)$. Combining this with the bound from Step 3 gives
\[
\mathcal E_D(t)\ge \frac{4}{3}\sqrt{\pi S_q(t-\delta)}
\qquad\text{for every }\delta\in(0,t).
\]
Letting $\delta\downarrow0$ proves \eqref{eq:radius_lower_general}.

Under Assumption \ref{ass:interior_dispersion}, we apply the generic bound by selecting the index set $\mathcal{J} = J_m$ and assigning the uniform radius $\textcolor{black}{r_s = \zeta|D|^{1/2} m^{-1/2}}$ to all selected atoms. This yields a selected cardinality $q = |J_m| \ge \eta m$ and an aggregate radius $S_q = q\zeta |D|^{1/2}m^{-1/2} \ge \eta\zeta\sqrt{|D|m}$. For every $t \in (0, S_q/(2m)]$, the generic lower bound guarantees$$\mathcal E_D(t) \ge \frac{4}{3}\sqrt{\pi S_q t} \ge \frac{4}{3}\sqrt{\pi\eta\zeta}|D|^{1/4}m^{1/4}\sqrt t.$$Because the upper limit of the feasible target distance satisfies $S_q/(2m) \ge \eta\zeta |D|^{1/2}/ (2\sqrt m) = \overline{t}_{\mathrm{low}}$, this global envelope bound holds uniformly over the entire interval $(0, \overline{t}_{\mathrm{low}}]$. 
\halmos
\begin{Remark}
\label{rem:lower_global_monotone}
Although Proposition~\ref{prop:radius_lower_general} constructs the lower bound only on the local interval $t\in(0,\bar t_{\mathrm{low}}]$, the envelope $\mathcal E_D(t)$ is nondecreasing in $t$ because the feasible Wasserstein ball expands with the radius. Therefore, under Assumption~\ref{ass:interior_dispersion},
\[
\mathcal E_D(t)\ge \frac{4}{3}\sqrt{\pi\eta\zeta}|D|^{1/4}m^{1/4}\sqrt{\min\{t,\bar t_{\mathrm{low}}\}},
\qquad 0<t\le \Delta_D.
\]
We focus on the local form over $(0,\bar t_{\mathrm{low}}]$, since that is the regime relevant for the interior scaling analysis and the subsequent lower bound on $\rho_\tau(\mathcal{L})$.
\end{Remark}
Applying the general construction of Proposition \ref{prop:radius_lower_general} to the geometrically regular atoms specified in Assumption \ref{ass:interior_dispersion}, we obtain an explicit lower bound for the envelope $\mathcal E_D(t)$. Mirroring our earlier upper-bound analysis, this immediately translates into a closed-form lower bound for the tardiness index in Corollary \ref{cor:kappa_lower}.

\begin{Corollary}
\label{cor:kappa_lower}
Suppose Assumption \ref{ass:interior_dispersion} holds, and define the constant $A_{\mathrm{low}}:=\frac{4}{3}\beta\sqrt{\pi n\eta\zeta}|D|^{1/4}m^{1/4}$. Then, the tardiness index is bounded below by
\begin{equation}
\label{eq:kappa_lower_scalar}
\rho_\tau(\mathcal{L})
\ge
\sup_{0<t\le \overline{t}_{\mathrm{low}}}
\frac{\bigl(A_{\mathrm{low}}\sqrt t-\tau\bigr)^+}{t}.
\end{equation}
Moreover, this supremum evaluates to the piecewise function
\begin{equation}
\label{eq:kappa_lower_piecewise}
\sup_{0<t\le \overline{t}_{\mathrm{low}}}
\frac{\bigl(A_{\mathrm{low}}\sqrt t-\tau\bigr)^+}{t}
=
\begin{cases}
\dfrac{A_{\mathrm{low}}^2}{4\tau}, & 0<\tau\le \dfrac{A_{\mathrm{low}}\sqrt{\overline{t}_{\mathrm{low}}}}{2},\\
\dfrac{A_{\mathrm{low}}\sqrt{\overline{t}_{\mathrm{low}}}-\tau}{\overline{t}_{\mathrm{low}}}, & \dfrac{A_{\mathrm{low}}\sqrt{\overline{t}_{\mathrm{low}}}}{2}<\tau<A_{\mathrm{low}}\sqrt{\overline{t}_{\mathrm{low}}},\\
0, & \tau\ge A_{\mathrm{low}}\sqrt{\overline{t}_{\mathrm{low}}}.
\end{cases}
\end{equation}
In particular, whenever $0 < \tau \le A_{\mathrm{low}}\sqrt{\overline{t}_{\mathrm{low}}}/2$,
\begin{equation}
\label{eq:kappa_lower_interior}
\rho_\tau(\mathcal{L})
\ge
\frac{A_{\mathrm{low}}^2}{4\tau}
=
\frac{4\pi\beta^2 n\eta\zeta}{9\tau}\sqrt{|D|m}.
\end{equation}
\end{Corollary}

\textit{Proof.} Substituting the envelope bound established in Proposition \ref{prop:radius_lower_general} into the variational representation of $\rho_\tau(\mathcal{L})$ given in \eqref{eq:kappa_radius} immediately establishes the inequality \eqref{eq:kappa_lower_scalar}. The subsequent scalar maximization problem is structurally identical to the one solved in Corollary \ref{cor:kappa_upper}, with the upper-bound parameters $A_{\mathrm{up}}$ and $\Delta_D$ systematically replaced by their lower-bound counterparts $A_{\mathrm{low}}$ and $\overline{t}_{\mathrm{low}}$. Consequently, the identical algebraic derivation yields the piecewise supremum \eqref{eq:kappa_lower_piecewise} and its explicit interior evaluation \eqref{eq:kappa_lower_interior}. \halmos

Corollary~\ref{cor:kappa_lower} complements the upper bound in Corollary~\ref{cor:kappa_upper} by showing that, under interior dispersion of the empirical support, the tight-target regime admits the same order of dependence on \(n\), \(m\), \(|D|\), and \(\tau\). Particularly, the lower bound scales with the parameters as $\frac{\beta^2 n}{\tau}\sqrt{|D|m}$, while the geometric regularity parameters only affect the multiplicative constant. Therefore, we can establish that the upper and lower bounds match in order when the maximizing Wasserstein radius remains interior. We next formalize this scaling rule and extend it to multi-vehicle tardiness indices.

\subsection{Interior-Regime Scaling Laws} 
The preceding bounds imply matching upper and lower orders for $\mathcal E_D(t)$ in the geometrically regular local regime, and thus extend to $\rho_\tau(\mathcal{L})$ whenever the unconstrained maximizer remains interior. Specifically, Proposition~\ref{prop:radius_upper} demonstrates that $\mathcal E_D(t)$ is bounded above by $\mathcal{O}(|D|^{1/4}m^{1/4}\sqrt{t})$, while Proposition~\ref{prop:radius_lower_general} confirms that this order is tightly attained from below under the geometric regularity of Assumption~\ref{ass:interior_dispersion}. Substituting this order into the representation \eqref{eq:kappa_radius} shows that, whenever the unconstrained maximizer remains in the interior of the feasible $t$-interval, the tardiness index satisfies
\[
\rho_\tau(\mathcal{L})=\Theta\left(\frac{\beta^2 n\sqrt{|D|m}}{\tau}\right),
\]
\noindent which scales linearly with the demand level $n$, increases with the empirical sample size $m$ at rate $\sqrt{m}$, increases with the zone area $|D|$ at rate $|D|^{1/2}$, and decreases inversely with the satisficing target $\tau$. Beyond this regime, the optimal radius reaches the feasible boundary, leading to a transition that is captured by the preceding piecewise upper and lower bounds. We numerically validate the scaling laws in \ref{ec:validation} of the Appendix.

This scaling characterization can also be extended to a multi-vehicle setting when the objective is to minimize the maximum routing time (e.g., the delivery system's makespan). Specifically, as shown in \cite{carlsson2025redesigning, carlsson2026equitable}, the minimax routing solution (that also minimizes the makespan) for $K$ identical vehicles would split the total routing time evenly, so the routing time of each vehicle is approximated by $\mathcal{L}_K(f)\approx \frac{\beta\sqrt n}{K}\int_D \sqrt{f(x)} dx $. Let \(\rho_\tau^{K}(\mathcal{L})\) denote the tardiness index for a \(K\)-vehicle routing system under a makespan target, leading to the following scaling result. 

\begin{Corollary}
\label{cor:kappa_multi_vehicle_scaling}
In the $K$-vehicle makespan minimizing setting, whenever the unconstrained maximizer in the analogue of \eqref{eq:kappa_radius} remains in the interior of the feasible $t$-interval, we have
\[
\rho_\tau^{K}(\mathcal{L})=\Theta\left(\frac{\beta^2 n\sqrt{|D|m}}{K^2\tau}
\right).
\]
\end{Corollary}

\textit{Proof.}
Under the balanced $K$-vehicle leading-order property, the optimal VRP makespan can be expressed as
\[
\frac{\beta\sqrt n}{K}\int_D \sqrt{f(x)} dx.
\]
Accordingly, the coefficient $\beta\sqrt n$ in \eqref{eq:kappa_radius} is replaced by $\beta\sqrt n/K$, so that
\[
\rho_\tau^{K}(\mathcal{L})=\sup_{0<t\le \Delta_D}
\frac{\left(\frac{\beta\sqrt n}{K}\mathcal E_D(t)-\tau\right)^+}{t}.
\]
The claim then follows from the same interior-regime scaling argument as above, with $\beta\sqrt n$ replaced by $\beta\sqrt n/K$.\halmos

Corollary~\ref{cor:kappa_multi_vehicle_scaling} highlights the fleet-size effect on the delivery time target vulnerability. Growing the number of vehicles reduces not only the nominal route length but also the sensitivity of target violation to spatial distributional shifts, yielding a \(1/K^2\) reduction in the interior-regime tardiness index. This suggests increasing marginal returns in improving the system reliability from expanding the fleet. 

\subsection{System Design Implications: Partition Guidelines and Insights}
We now discuss how the closed-form scaling laws for the TSP tardiness index can inform region partition guidelines for delivery systems. Region partitioning, or districting, is a common tactical design decision in last-mile logistics \citep{banerjee2022fleet, guo2023towards}. The goal is to divide the service region into subregions, or zones, each managed by a dedicated delivery station with a given fleet of vehicles. The dedicated zones facilitate practical management and foster the accumulation of drivers' local tacit knowledge in the assigned territory. 

Without loss of generality, consider a partitioning task that aims to divide the service region \(D\) into two zones $D_1$ and $D_2$. Zone $D_i$ dispatches $K_i$ vehicles subject to a delivery time target $\tau_i$. Suppose historical and realized demand are uniformly distributed over \(D\), which is assumed for analytical convenience and can hold for fully built-out residential neighborhoods, then the sample demand size and realized demand size in each zone are proportional to its area, implying $m_i=m\frac{|D_i|}{|D|}, \ n_i=n\frac{|D_i|}{|D|}, i=1,2$. We consider two objectives. The first is aggregate tardiness-risk minimization, which concerns the sum of the two zone-level tardiness indices. The second is worst-zone tardiness-risk minimization, which aims to control the larger of the two zone-level tardiness indices. Proposition~\ref{cor:two_zone_area_ratio} presents the closed-form partition rules in the interior regimes of these two objectives.

\begin{Proposition}
\label{cor:two_zone_area_ratio}
Given the leading-order approximation $
\rho_{\tau_i}^{(K_i,D_i)}(\mathcal{L})\asymp \frac{\beta^2 n_i\sqrt{|D_i|m_i}}{K_i^2\tau_i},$ $i=1,2,$ the optimal partitions that satisfy the coverage constraint (i.e., $|D_1|+|D_2|=|D|$) under aggregate and worst-zone tardiness-risk objectives satisfy
\begin{enumerate}
    \item[(i)] Aggregate risk minimization (minimizing $\sum_{i=1}^2 \rho_{\tau_i}^{(K_i,D_i)}(\mathcal{L})$): $
    |D_1|/|D_2|=K_1^2\tau_1/K_2^2\tau_2.$
    \item[(ii)] Worst-zone risk minimization (minimizing $\max_{i=1,2} \rho_{\tau_i}^{(K_i,D_i)}(\mathcal{L})$):
    $|D_1|/|D_2|=K_1\sqrt{\tau_1}/K_2\sqrt{\tau_2}$.

\end{enumerate}
\end{Proposition}
\textit{Proof.} Substituting $m_i=m(|D_i|/|D|)$ and $n_i=n(|D_i|/|D|)$ yields
\[
\rho_\tau^{(K_i,D_i)}(\mathcal{L})\asymp \frac{\beta^2 n\sqrt m}{|D|^{3/2}}\cdot \frac{|D_i|^2}{K_i^2\tau_i},\qquad i=1,2.
\]
Omitting the common positive constant, the optimization problems reduce to minimizing $\sum_{i=1}^2 \frac{|D_i|^2}{K_i^2\tau_i}$ for aggregate control, and minimizing $\max_{i=1,2} \left\{\frac{|D_i|^2}{K_i^2\tau_i}\right\}$ for worst-zone control, subject to $|D_1|+|D_2|=|D|$. For (i), the objective is strictly convex. The first-order necessary and sufficient condition yields $\frac{|D_1|}{K_1^2\tau_1}=\frac{|D_2|}{K_2^2\tau_2}$. For (ii), at the optimum, the two terms must be equal; otherwise, shifting area from the larger term to the smaller one strictly reduces the maximum. Equating the terms yields $\frac{|D_1|^2}{K_1^2\tau_1}=\frac{|D_2|^2}{K_2^2\tau_2}$. Both conditions immediately provide the stated optimal ratios. \halmos

\vspace{5pt}

Proposition~\ref{cor:two_zone_area_ratio} establishes simple principles for partitioning a region under different risk objectives. For aggregate target-risk minimization, the partition equalizes the marginal contribution of area to the total tardiness index, significantly allocating more area to zones with a larger fleet size or more relaxed targets (scaled by $K_i^2\tau_i$). In contrast, worst-zone control protects the most fragile zone by equalizing the absolute tardiness indices. The corresponding partition responds less aggressively to differences in vehicle capacity and delivery targets (scaled by $K_i\sqrt{\tau_i}$). When the same delivery target is imposed across zones, the two partition principles allocate areas based on the fleet size, e.g., the minimax partition allocates the delivery area proportionally to the fleet size. 

\begin{Remark}
    For nonuniform demand and more than two zones, the tardiness-risk objective can be optimized over using existing partitioning techniques such as Voronoi diagrams \citep{carlsson2025redesigning}. Specifically, given fixed depots or delivery stations, the partition can be controlled by a continuous zone-weight vector, which can then be optimized using a cutting plane or gradient descent algorithm \citep{carlsson2018wasserstein, carlsson2025redesigning}. Our proposed SOCP formulation can be called directly in those algorithms as a subroutine.
\end{Remark}

\section{Numerical Experiments}
In this section, we conduct numerical experiments to demonstrate the value of the tardiness index using synthetic and real-world data. The algorithms were implemented in Python using Mosek 11 for SOCP. All numerical experiments were conducted on a Windows 10 machine with a 64-bit operating system and 16GB of RAM. 

\subsection{Value of Tardiness Index} We examine the impact of the TSP tardiness index on the operational risk metrics with both synthetic data and real-world Amazon last-mile data \citep{merchan20242021}. The index is computed from a calibration sample and is then compared with out-of-sample overtime outcomes, including average overtime, 95\% Conditional Value-at-Risk (CVaR) of overtime, and overtime rate. These metrics are practically easy to interpret but hard to directly optimize over under locational uncertainty. 
In the synthetic experiment, we employ four demand distributions: uniform, beta, kernel-based clustering, and Gaussian mixture. For each distribution family, we generate 10 distinct zones, each with 35 daily demand realizations.  The detailed data generating process is presented in \ref{ec:syn_details} of the Appendix. Table~\ref{tab:sim_rho_overtime_by_distribution} summarizes the correlation coefficient, $R^2$, and $p$-value between the zone-level tardiness index and the considered operational risk metrics.

\begin{table}[htbp]
\centering
\caption{Zone-level association between $\rho^\star$ and overtime risk in synthetic simulations}
\label{tab:sim_rho_overtime_by_distribution}
\begin{adjustbox}{width=\textwidth}
\begin{tabular}{lccccccccc}
\toprule
\multirow{2}{*}{Demand distribution}
& \multicolumn{3}{c}{Average overtime}
& \multicolumn{3}{c}{95\% CVaR of overtime}
& \multicolumn{3}{c}{Overtime rate} \\
\cmidrule(lr){2-4}\cmidrule(lr){5-7}\cmidrule(lr){8-10}
& Correlation & $R^2$ & $p$-value
& Correlation & $R^2$ & $p$-value
& Correlation & $R^2$ & $p$-value \\
\midrule
Uniform
& 0.966 & 0.933 & $<0.001^{***}$
& 0.637 & 0.406 & $0.048^{**}$
& 0.831 & 0.690 & $0.003^{***}$ \\

Kernel
& 0.954 & 0.911 & $<0.001^{***}$
& 0.912 & 0.831 & $<0.001^{***}$
& 0.832 & 0.692 & $0.003^{***}$ \\

GMM + noise
& 0.958 & 0.917 & $<0.001^{***}$
& 0.955 & 0.913 & $<0.001^{***}$
& 0.876 & 0.767 & $<0.001^{***}$ \\

Beta
& 0.981 & 0.962 & $<0.001^{***}$
& 0.960 & 0.921 & $<0.001^{***}$
& 0.805 & 0.649 & $0.005^{***}$ \\
\bottomrule
\end{tabular}
\end{adjustbox}
\begin{tablenotes}
\footnotesize \item $^{***}p<0.01$;$^{**}p<0.05$;$^{*}p<0.10$.
\end{tablenotes}
\end{table}





We observe that the estimated tardiness index is consistently associated with out-of-sample overtime risk across all four demand families. The association is most stable for average overtime: the correlations are all above 0.95, the corresponding \(R^2\) values exceed 0.91, and all \(p\)-values are below 0.001. These results indicate that \(\rho^\star\) provides a reliable assessment of zones by their expected target-violation severity. The positive association also extends to tail- and frequency-based risk measures. For the 95\% CVaR of overtime, the relationship is strong for the nonuniform-demand families and remains statistically significant in the uniform case. For the overtime rate, all correlations are high and significant at the 1\% level. These results suggest that \(\rho^\star\) captures not only expected target-violation severity but also broader overtime-risk patterns, particularly when spatial heterogeneity is present. We further extend the evaluation to real-world data from two representative Amazon delivery stations; details are in \ref{EC:amazon}. While the Amazon data yield noisier estimates (partly due to unobservable factors), the tardiness index remains indicative of the overtime risk after accounting for service-time variability.

\subsection{Comparison of Partition Rules} 
The scaling analysis presented in Section~\ref{sec:scaling} establishes two closed-form partition rules. In this section, we report their efficacy via out-of-sample tests under stochastic demand realizations, where actual routing durations are computed by solving exact TSP instances. Specifically, we assume (i) demand size is drawn from a truncated Poisson distribution with a mean of 50 and support $[30, 80]$, (ii) demand locations are generated from four heterogeneous spatial distributions (uniform, kernel, Gaussian-mixture, and beta), and (iii) realized routing times are computed by solving the within-zone TSPs by Google OR-Tools. The evaluation reports the aggregate overtime (summation of the overtime across the zones) and maximum overtime metrics. 

We denote the proposed aggregate-risk and worst-zone-risk partition rules by \(R_T\) and \(R_M\), respectively. We compare them with three alternative partition rules that disregard distributional uncertainty: (i) Violation-total ($D_T$) that minimizes the total violation of the target; (ii) Violation-max ($D_M$) that minimizes the maximum violation of the target; and (iii) Length-max ($L_M$) that minimizes the maximum route length. The detailed definition,  implementation, and evaluation results are presented in \ref{ec:compare} of the Appendix.  We observe that the two tardiness-index partition rules have complementary strengths. For aggregate overtime metrics, \(R_T\) is consistently competitive and performs particularly well on risk-sensitive outcomes, including the 95\% CVaR of aggregate overtime and the aggregate overtime rate. 
For worst-case overtime metrics, \(R_M\) is most effective in controlling tail severity, achieving the smallest gap for the 95\% CVaR of maximum overtime. This aligns with its minimax construction, which balances the two zone-level tardiness indices and protects the more fragile zone. These results suggest that, despite their simplicity, the proposed partition rules exhibit superior risk-control performance. Notably, they achieve both competitive average and tail performances. 

\section{Concluding Remarks}
In this paper, we study the tardiness of TSP solutions using a robust satisficing model. We show that the TSP route target fragility, as measured by the TSP tardiness index, can grow linearly with the realized demand size, in contrast to the conventional square-root rate. The proven scaling law also links the tardiness index to the number of empirical support points, which is tied to the problem's spatial features. In addition to theoretical insights, we provide efficient computational tools and partition rules that facilitate the practical use of this index in managing routing systems. Because the derived structural properties concern locational uncertainty, future work can investigate tardiness arising from multiple sources of uncertainty, including travel and service times. It would also be valuable to extend the analysis to VRPs with heterogeneous vehicle capacities.


\bibliographystyle{informs2014trsc}
\bibliography{mybib}

\ECSwitch
\counterwithin{table}{section}
\counterwithin{figure}{section}
\numberwithin{equation}{section}  
\makeatletter

\renewcommand{\p@subfigure}{} 
\renewcommand{\thesubfigure}{\thesection.\arabic{figure}(\alph{subfigure})}
\makeatother
{\center \Large Appendix for ``Traveling Salesman Tardiness"}

\section{Supporting Numerical Results}
\subsection{Numerical Validation Results} \label{ec:validation}
Without loss of generality, we assume the historical demand locations are uniformly distributed over a unit-area square, and set the integration grid size to $50\times 50$. We take $(m,n,\tau)=(50,50,3)$ as the baseline configuration. Starting from this baseline, we vary two of the three parameters while fixing the remaining one at its baseline value. The TSP tardiness indices corresponding to these settings are computed using the proposed SOCP formulation. Figures~\ref{fig:fragility_a}--\ref{fig:fragility_f} illustrate how the tardiness index responds to changes in $m$, $n$, and $\tau$. The numerical pattern aligns with our theoretical bound analysis. Specifically, for fixed \(m\) and \(\tau\), the tardiness index increases approximately linearly with \(n\); for fixed \(n\) and \(\tau\), the index increases with \(m\), but at a visibly sublinear rate, which is consistent with the \(\sqrt m\)-type dependence predicted by the bounds. Moreover, the dependence on \(\tau\) is piecewise: the tardiness index decreases sharply when the target is tight, then decreases more gradually, and eventually reaches zero once the target becomes sufficiently relaxed. 
\begin{figure}[htbp]
    \centering
    \subfigure[$\tau=3$, x-axis: $n$]{
        \includegraphics[width=0.48\textwidth]{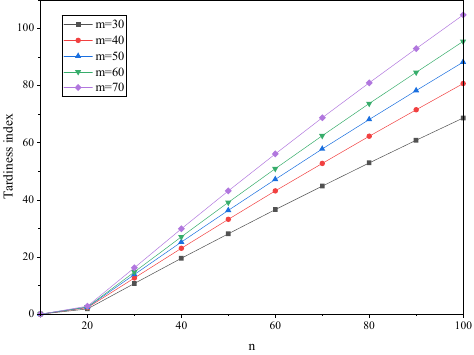}
        \label{fig:fragility_a}
    }%
    \hfill
    \subfigure[$m=50$, x-axis: $n$]{
        \includegraphics[width=0.48\textwidth]{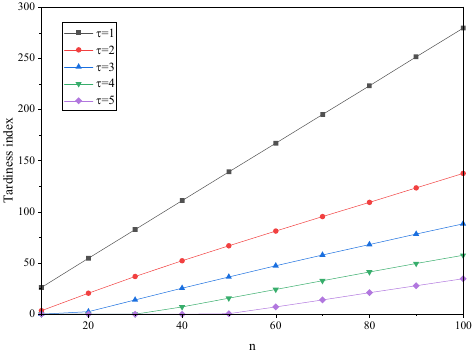}
        \label{fig:fragility_b}
    }%

   \vspace{-0.5em}

    \subfigure[$\tau=3$, x-axis: $m$]{
        \includegraphics[width=0.48\textwidth]{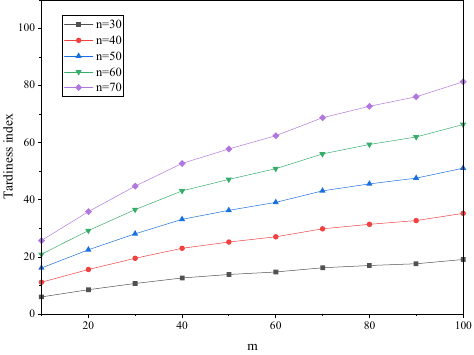}
        \label{fig:fragility_c}
    }%
    \hfill
    \subfigure[$n=50$, x-axis: $m$]{
        \includegraphics[width=0.48\textwidth]{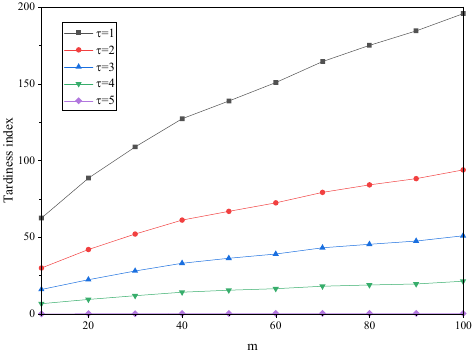}
        \label{fig:fragility_d}
    }%

    \vspace{-0.5em}

    \subfigure[$n=50$, x-axis: $\tau$]{
        \includegraphics[width=0.48\textwidth]{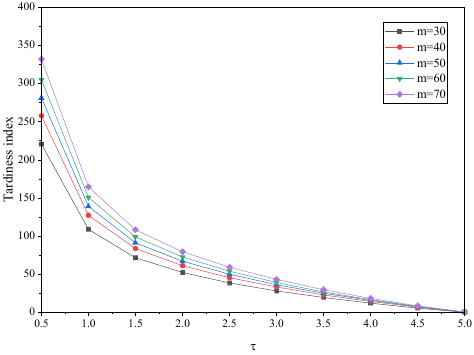}
        \label{fig:fragility_e}
    }%
    \hfill
    \subfigure[$m=50$, x-axis: $\tau$]{
        \includegraphics[width=0.48\textwidth]{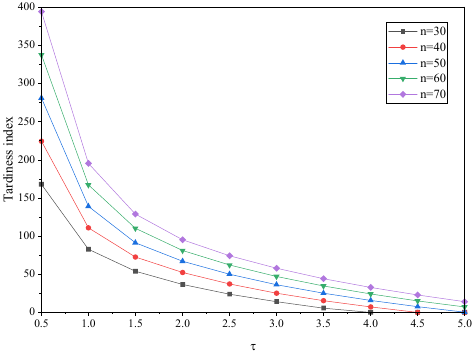}
        \label{fig:fragility_f}
    }%
    \caption{Sensitivity of the tardiness index to $m$, $n$, and $\tau$.}
    \label{fig:fragility}
\end{figure}
\subsection{Computational Performance Evaluation}\label{EC:computation}
We compare the computational efficiency of the analytic center cutting plane method (ACCPM), implemented based on Algorithm~1 of \citetapp{carlsson2018wasserstein}, and the direct SOCP approach for solving the single-zone DRO problem under a fixed Wasserstein radius~$t$. Specifically, we randomly generate historical demand locations with uniform probabilities in a unit square region, and we set the Wasserstein radius $t$ as 0.1. Since the efficiency of the DRO model is independent of the realized demand size $n$, Table~\ref{tab:time_comparison} reports the computational performance with varying empirical demand size $m$ and the grid resolution used to discretize the continuous integral. For each instance class, we report the average runtime of ACCPM and SOCP, in seconds, together with the relative time improvement $\frac{T_{\mathrm{ACCPM}}-T_{\mathrm{SOCP}}}{T_{\mathrm{ACCPM}}}\times 100\%$. All values are averaged over 10 randomly generated instances. Across all combinations of $m$ and grid size, the computational gap becomes particularly pronounced for moderate to large instances, where SOCP often achieves runtime reductions of roughly two orders of magnitude relative to ACCPM. This advantage is also reflected in the percentage reductions reported in the table, which exceed $90\%$ for most moderate and large instances and frequently exceed $99\%$. Overall, the computational burden of ACCPM increases much more sharply with problem size and grid resolution, and SOCP reformulation offers markedly better scalability.
\begin{table}[htbp]
\centering
\caption{Computational time comparison for the single-zone DRO problem}
\label{tab:time_comparison}
\resizebox{0.9\textwidth}{!}{
\begin{tabular}{c ccc ccc ccc}
\toprule
\multirow{2}{*}{$m$}
& \multicolumn{3}{c}{$30\times 30$ grid}
& \multicolumn{3}{c}{$50\times 50$ grid}
& \multicolumn{3}{c}{$100\times 100$ grid} \\
\cmidrule(lr){2-4} \cmidrule(lr){5-7} \cmidrule(lr){8-10}
& ACCPM & SOCP & Impr. (\%)
& ACCPM & SOCP & Impr. (\%)
& ACCPM & SOCP & Impr. (\%) \\
\midrule
$20$                 & 2.96                      & 0.08                     & 97.42                        & 2.79                      & 0.27                     & 90.15                        & 4.20                      & 1.42                     & 66.13                        \\
$40$                 & 13.57                     & 0.22                     & 98.36                        & 15.82                     & 1.09                     & 93.09                        & 41.87                     & 4.45                     & 89.36                        \\
$60$                 & 82.40                     & 0.31                     & 99.62                        & 109.47                    & 1.16                     & 98.94                        & 126.84                    & 7.26                     & 94.27                        \\
$80$                 & 857.98                    & 0.38                     & 99.96                        & 384.29                    & 1.31                     & 99.66                        & 486.97                    & 6.40                     & 98.69                        \\
$100$                & 633.34                    & 0.46                     & 99.93                        & 688.10                    & 1.33                     & 99.81                        & 935.89                    & 6.75                     & 99.28                  \\
\bottomrule
\end{tabular}
}
\end{table}

\section{Additional Details about the Experiments}
\subsection{Synthetic Experiment Setup} \label{ec:syn_details}
In the synthetic experiment, the daily demand volume for each zone is drawn from a truncated normal distribution with a zone-specific mean ranging between 80 and 100, while customer locations are sampled from the corresponding spatial distribution. For every zone-day instance, we compute the realized routing time by solving a Euclidean TSP using OR-Tools. Based on these realized TSP lengths, we also estimate the $\beta$ parameter of the BHH formula for each individual zone. 
We pool the realized daily TSP lengths across all zones and days within family $d$ to define the target as $\tau^{(d)} = c \cdot Q_q\left(\{L_{z,t}^{(d)}: z=1,\ldots,10,\ t=1,\ldots,35\}\right)$. Here, $L_{z,t}^{(d)}$ denotes the TSP length of zone $z$ on day $t$, and $Q_q(\cdot)$ represents the empirical $q$-quantile. In our computational setup, we set the scaling factor $c=0.95$ and the quantile $q = 0.45$. To evaluate out-of-sample performance, we apply a randomized cross-validation approach. For each zone, we generate 10 independent train-test splits, allocating 5 days to training and 30 days to testing. The tardiness index $\rho$ is calibrated on the training data, whereas operational risk metrics are evaluated on the test set. Finally, we average $\rho$ and the corresponding risk metrics across the 10 splits for each zone.

\subsection{Amazon Dataset: Analysis and Evaluation Results}\label{EC:amazon}
We examine the 18 stations recorded in the Amazon last-mile dataset \citepapp{merchan20242021}, noting substantial variation in data coverage, specifically, the number of zones and the number of active days with recorded routes. As detailed in Table~\ref{tab:amazon_station_zone_coverage}, columns \textit{$Z$} denote the total number of zones, columns \textit{Days} report the range of active days across these zones, and columns \textit{$Z_{25}$} indicate the number of zones with at least 25 active days. The stations are sorted by $Z_{25}$. We focus on DLA8 and DLA9 because they have the largest number of zones with sufficient active-day coverage, enabling a more stable station-zone-level assessment of the association between the tardiness-fragility index and out-of-sample performance. Specifically, DLA8 comprises 10 zones, and DLA9 comprises 9 zones, with each zone containing between 28 and 93 days of active route records. 

\begin{table}[htbp]
\centering
\caption{Zone-level coverage across Amazon stations}
\label{tab:amazon_station_zone_coverage}
\begin{tabular}{lrrr@{\qquad}lrrr}
\hline
Station & \(Z\) & Days & \(Z_{25}\) 
& Station & \(Z\) & Days & \(Z_{25}\) \\
\hline
DLA8 & 10 & 46--56 & 10 
& DSE4 & 4  & 44--97 & 4 \\
DLA9 & 9  & 28--93 & 9  
& DCH4 & 4  & 48--53 & 4 \\
DLA7 & 8  & 1--80  & 7  
& DCH1 & 6  & 14--58 & 4 \\
DLA4 & 9  & 1--48  & 7  
& DSE5 & 3  & 60--77 & 3 \\
DBO3 & 6  & 38--66 & 6  
& DLA5 & 5  & 6--48  & 3 \\
DCH3 & 5  & 47--57 & 5  
& DCH2 & 16 & 2--30  & 3 \\
DAU1 & 8  & 3--48  & 5  
& DBO2 & 2  & 66--79 & 2 \\
DLA3 & 16 & 1--55  & 5  
& DBO1 & 9  & 3--17  & 0 \\
DSE2 & 12 & 1--74  & 5  
& DBO6 & 1  & 1--1   & 0 \\
\hline
\end{tabular}
\end{table}

Consistent with our synthetic experimental design, we evaluate out-of-sample performance using randomized splits. For each zone, we randomly sample 5 days for calibration (training) and reserve the remaining days for testing. We establish a system-level service target of 8 hours, which closely aligns with the empirical average of 8.1 hours observed in the raw sample. Recognizing that individual zones exhibit heterogeneous road network structures and service-time compositions, we translate the global service target into zone-specific travel-time targets. Correspondingly, the BHH constant $\beta$ is calibrated independently for each zone using the calibration data exclusively. To ensure the robustness of our results, this calibration and evaluation procedure is repeated across 20 independent randomized splits per zone. The out-of-sample performance of both stations is summarized in Table~\ref{tab:rho_overtime_station}.


\begin{table}[htbp]
\centering
\caption{Zone-level association between $\rho^\star$ and overtime risk in Amazon last-mile data}
\label{tab:rho_overtime_station}
\begin{adjustbox}{width=0.95\textwidth}
\begin{tabular}{llcccccc}
\toprule
Performance measure & Station 
& Correlation & $R^2$ & $p$-value 
& Partial Correlation & $R^2$ with svc. & $p_{\rho^\star\mid \mathrm{svc}}$-value \\
\midrule
\multirow{2}{*}{Average overtime}
& DLA8 & 0.527 & 0.278 & 0.118 & 0.792 & 0.781 & $0.011^{**}$ \\
& DLA9 & 0.432 & 0.186 & 0.246 & 0.623 & 0.801 & $0.099^{*}$ \\
\midrule
\multirow{2}{*}{95\% CVaR of overtime}
& DLA8 & 0.289 & 0.083 & 0.418 & 0.484 & 0.557 & 0.186 \\
& DLA9 & 0.477 & 0.228 & 0.194 & 0.687 & 0.820 & $0.060^{*}$ \\
\midrule
\multirow{2}{*}{Overtime rate}
& DLA8 & 0.639 & 0.408 & $0.047^{**}$ & 0.798 & 0.720 & $0.010^{***}$ \\
& DLA9 & 0.518 & 0.268 & 0.153 & 0.875 & 0.941 & $0.004^{***}$ \\
\bottomrule
\end{tabular}
\end{adjustbox}
\end{table}

For each risk metric, we report both unconditional and service-time-controlled associations. We introduce this control because our tardiness index $\rho^\star$ is constructed from the spatial routing component, while the observed total time duration in the Amazon data includes both routing and stop-level service times. Notably, service time exhibits substantial cross-zone variation driven by on-site delivery conditions. The unconditional statistics, which are directly comparable to the simulation results, include the correlation with $\rho^\star$, the $R^2$ from a univariate regression on $\rho^\star$, and the corresponding $p$-value. The controlled statistics use the average training-set service time per route as an ex ante proxy for zone-level service difficulty and report the partial correlation with $\rho^\star$, the controlled-regression $R^2$, and the $p$-value of $\rho^\star$. Compared with the synthetic experiments, the unconditional associations in the Amazon data are weaker, reflecting the additional operational noise and cross-zone heterogeneity in real delivery operations. Nevertheless, the correlations are consistently positive across both stations and all three performance measures. More importantly, after controlling for training-set service time, the association between $\rho^\star$ and route overtime risk becomes substantially stronger for average overtime and overtime rate. These results suggest that service-time heterogeneity masks part of the spatial fragility signal in the raw data, while $\rho^\star$ remains informative after accounting for service difficulty.

\subsection{Details of the Comparison Analysis} \label{ec:compare}
We first provide the details of the three alternative partition rules. Let
\[
s:=\frac{|D_1|}{|D|},\qquad 
\bar L_1(s)=\frac{B_0}{K_1}s,\qquad 
\bar L_2(s)=\frac{B_0}{K_2}(1-s),
\]
where $B_0:=\beta\sqrt{n|D|}$. Here $\bar L_1(s)$ and $\bar L_2(s)$ denote the nominal routing time of the two zones. Table~\ref{tab:static_rules} formalizes the definition of the three rules.
\begin{table}[htbp]
    \centering
    \caption{Nominal benchmark partition rules}
    \label{tab:static_rules}
    \resizebox{\textwidth}{!}{
    \renewcommand{\arraystretch}{1.2} 
    \begin{tabular}{c c l}
    \toprule
    Name (Rule) & Objective & Closed-form share of Zone 1 \\
    \midrule
    
    \makecell{violation-total \\ ($D_T$)}
    & $\displaystyle \min_{0\le s\le 1}\Bigl(|\bar L_1(s)-\tau_1|+|\bar L_2(s)-\tau_2|\Bigr)$
    & $\displaystyle
    \begin{cases}
    \frac{\tau_1K_1}{B_0}, & K_1<K_2,\\[2mm]
    1-\frac{\tau_2K_2}{B_0}, & K_1>K_2,\\[2mm]
  \text{any }s \in 
\mathrm{conv}\left\{\frac{\tau_1K}{B_0}, 1-\frac{\tau_2K}{B_0}
\right\}\cap[0,1], & K_1=K_2
    \end{cases}$
    \\[8mm]

    \makecell{violation-max \\ ($D_M$)}
    & $\displaystyle \min_{0\le s\le 1}\max\Bigl\{|\bar L_1(s)-\tau_1|,\ |\bar L_2(s)-\tau_2|\Bigr\}$
    & $\displaystyle \frac{B_0K_1+K_1K_2(\tau_1-\tau_2)}{B_0(K_1+K_2)}$
    \\[6mm]


    \makecell{length-max \\ ($L_M$)}
    & $\displaystyle \min_{0\le s\le 1}\max\Bigl\{\bar L_1(s),\bar L_2(s)\Bigr\}$
    & $\displaystyle \frac{K_1}{K_1+K_2}$
    \\[2mm]
    
    \bottomrule
    \end{tabular}
    } 
\end{table}

For the partition-rule evaluation, we conduct the two-zone out-of-sample simulation on a $2\times 1$ rectangular service region. The system configurations are defined by three parameters: the fleet-size pair \((K_1,K_2)\in\{(2,2),(2,3),(3,2)\}\), the target ratio
\(r=\tau_1/\tau_2\in\{0.5,1,2\}\), and the aggregate target \(T\in\{2.9,3.2,3.5\}\). The ratio \(r\) controls the imbalance between the two zone-level targets, whereas \(T\) controls the overall target tightness. Given a pair $(r,T)$, the individual targets are set to $\tau_1 = \frac{Tr}{1+r}$ and $\tau_2 = \frac{T}{1+r}$. This design yields 27 fleet-target configurations, which are subsequently evaluated under the four aforementioned demand distributions, resulting in 108 unique experimental combinations. For each spatial distribution, customer locations are sampled without replacement from a fixed 200-point candidate pool, with selection probabilities strictly proportional to their corresponding spatial densities. Finally, we draw 100 independent replications for each of the 108 combinations to ensure statistical reliability.

Let $\mathrm{TSP}(S_i)$ denote the length of the optimal tour for the realized demand subset $S_i$ allocated to zone $i$. We define the normalized per-vehicle routing time as $L_i(S_i;K_i) = \mathrm{TSP}(S_i)/K_i$, which yields a zone-level overtime excess of $(L_i(S_i;K_i) - \tau_i)^+$. For each experimental replication, we compute the aggregate and maximum overtime excesses across the two zones as follows: $$E^{\mathrm{sum}} = \sum_{i=1}^2 \left(\frac{\mathrm{TSP}(S_i)}{K_i}-\tau_i\right)^+, \quad E^{\mathrm{max}} = \max_{i=1,2} \left(\frac{\mathrm{TSP}(S_i)}{K_i}-\tau_i\right)^+.$$

Figure~\ref{fig:partition} visualizes the performance gap to the best-performing rule across six out-of-sample metrics. The first two rows use relative gaps for overtime magnitude measures, while the third row uses absolute gaps in overtime rates, measured in percentage points. It shows that the proposed
tardiness-index rules remain close to the best-performing rule across most metrics. In particular, \(R_T\) is more competitive for aggregate overtime measures, while \(R_M\) is more effective for maximum-overtime and tail-risk measures. These patterns suggest that the tardiness index can provide a useful guide for delivery-zone partitioning.



\begin{figure}[htbp]
    \centering
    \subfigure[Average aggregate overtime (\%)]{
        \includegraphics[width=0.48\textwidth]{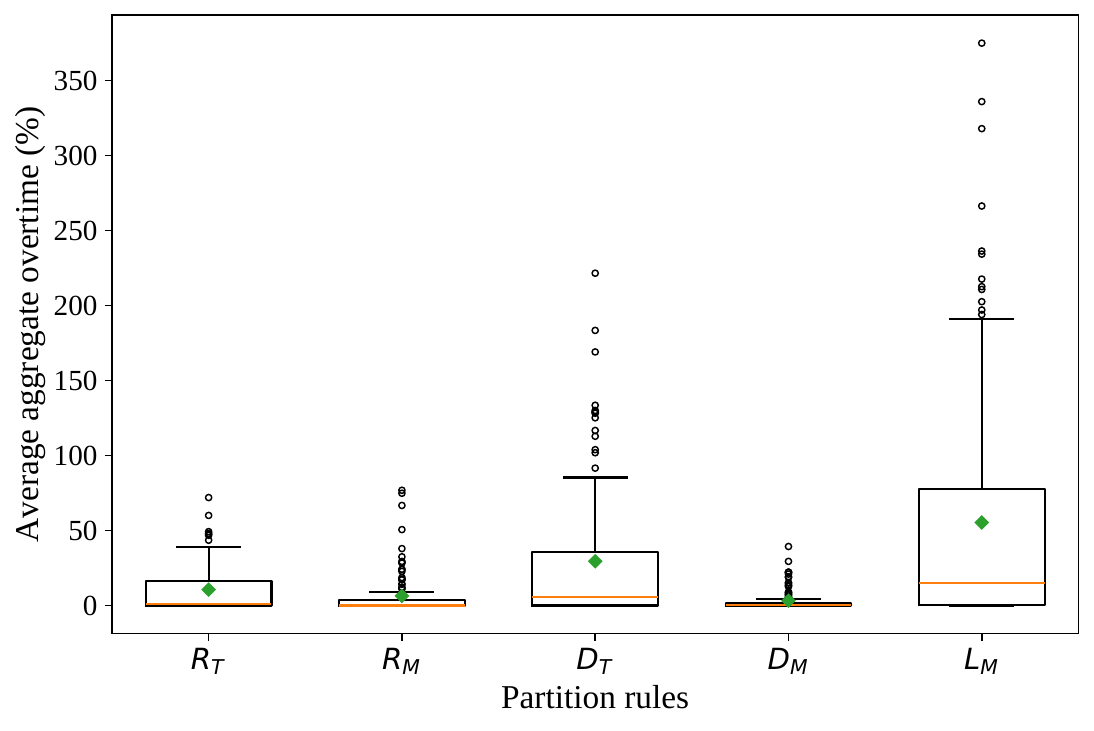}
        \label{fig:meanTE}
    }%
    \hfill
    \subfigure[Average maximum overtime (\%)]{
        \includegraphics[width=0.48\textwidth]{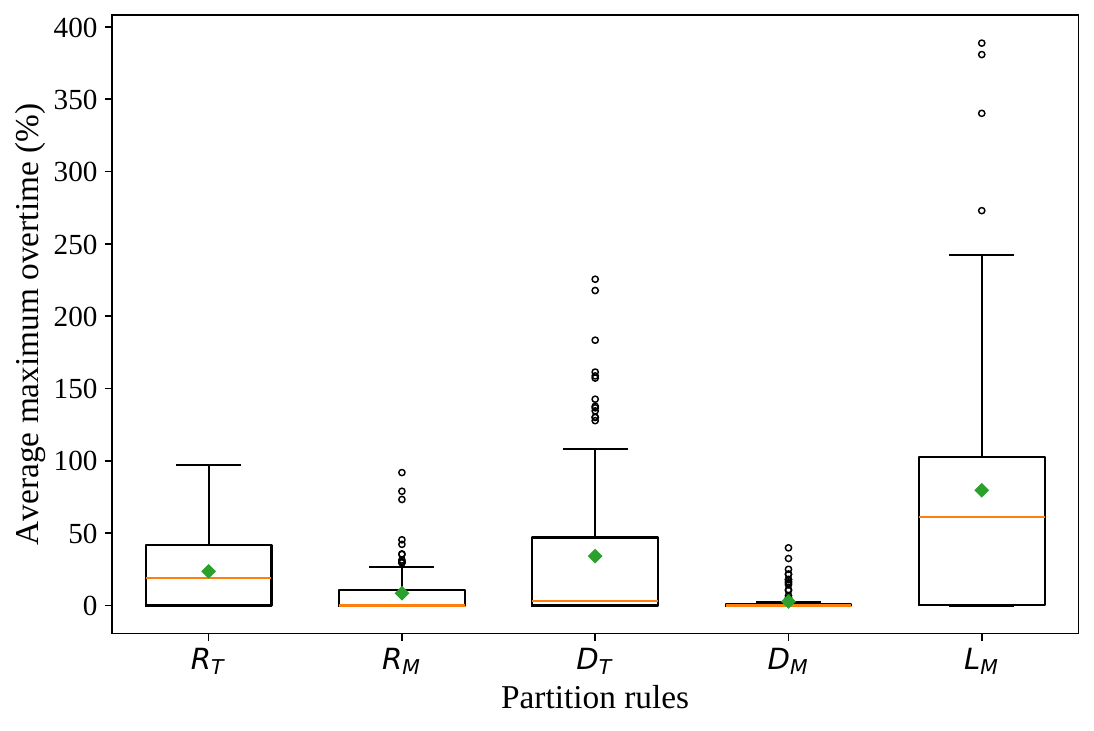}
        \label{fig:meanME}
    }%

    \vspace{-0.4em}

    \subfigure[95\% CVaR of aggregate overtime (\%)]{
        \includegraphics[width=0.48\textwidth]{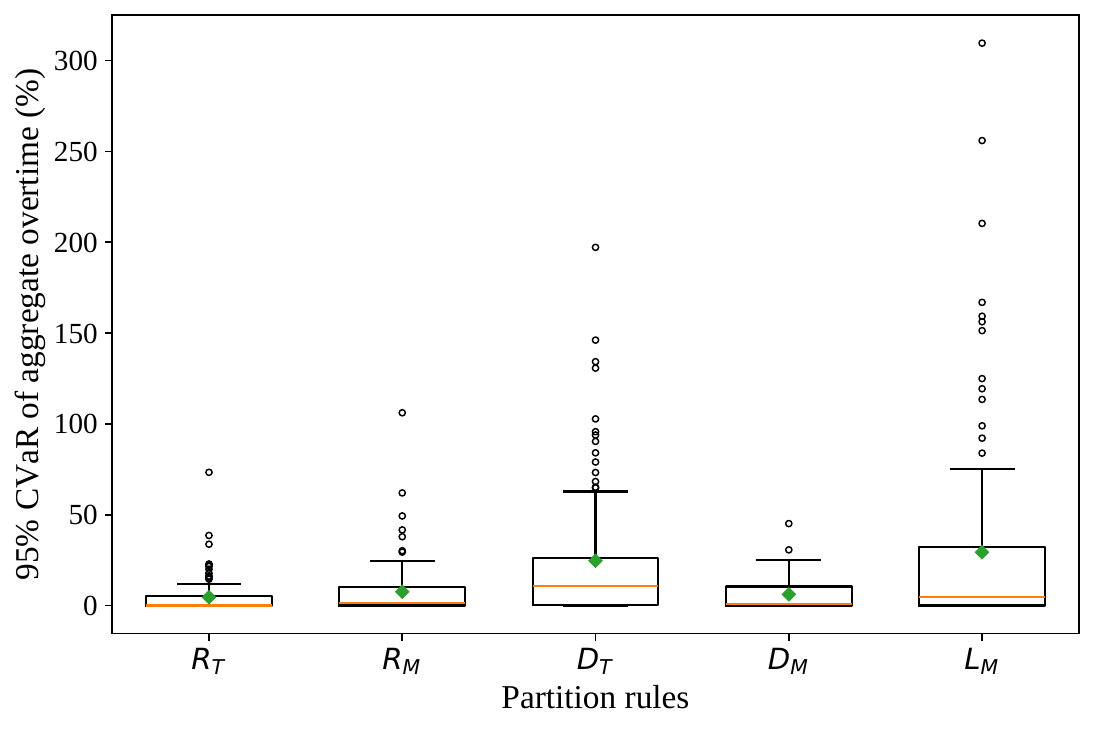}
        \label{fig:cvarTE}
    }%
    \hfill
    \subfigure[95\% CVaR of maximum overtime (\%)]{
        \includegraphics[width=0.48\textwidth]{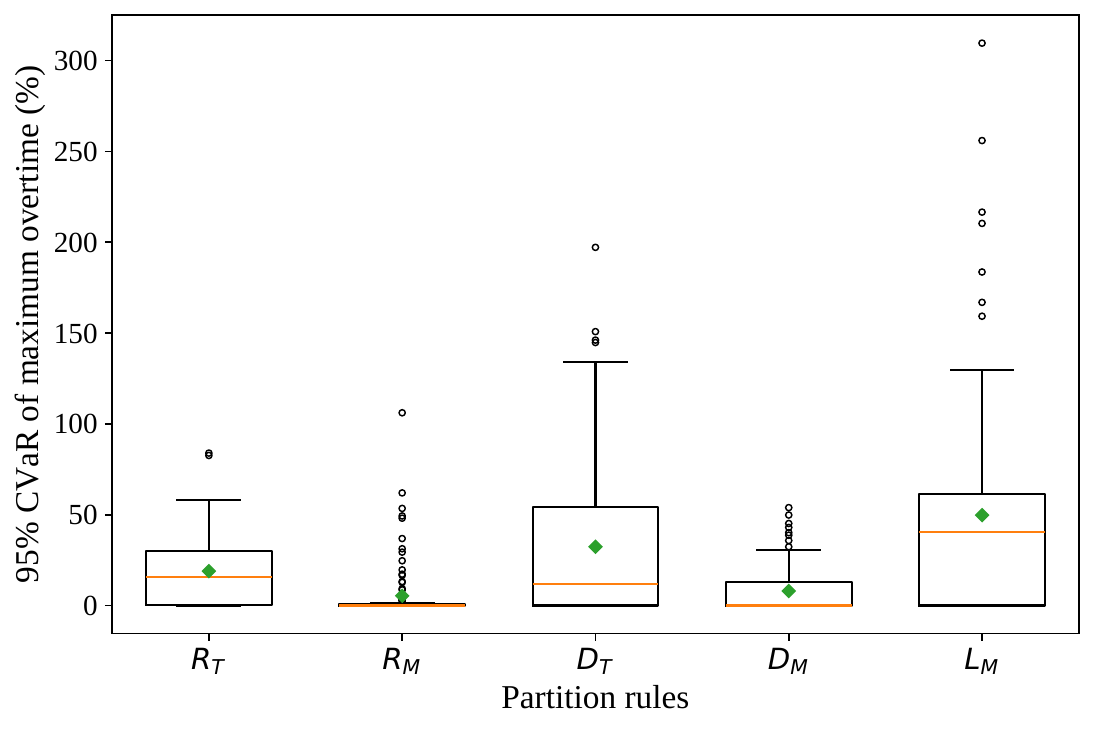}
        \label{fig:cvarME}
    }%

    \vspace{-0.4em}

    \subfigure[Aggregate overtime rate (pp)]{
        \includegraphics[width=0.48\textwidth]{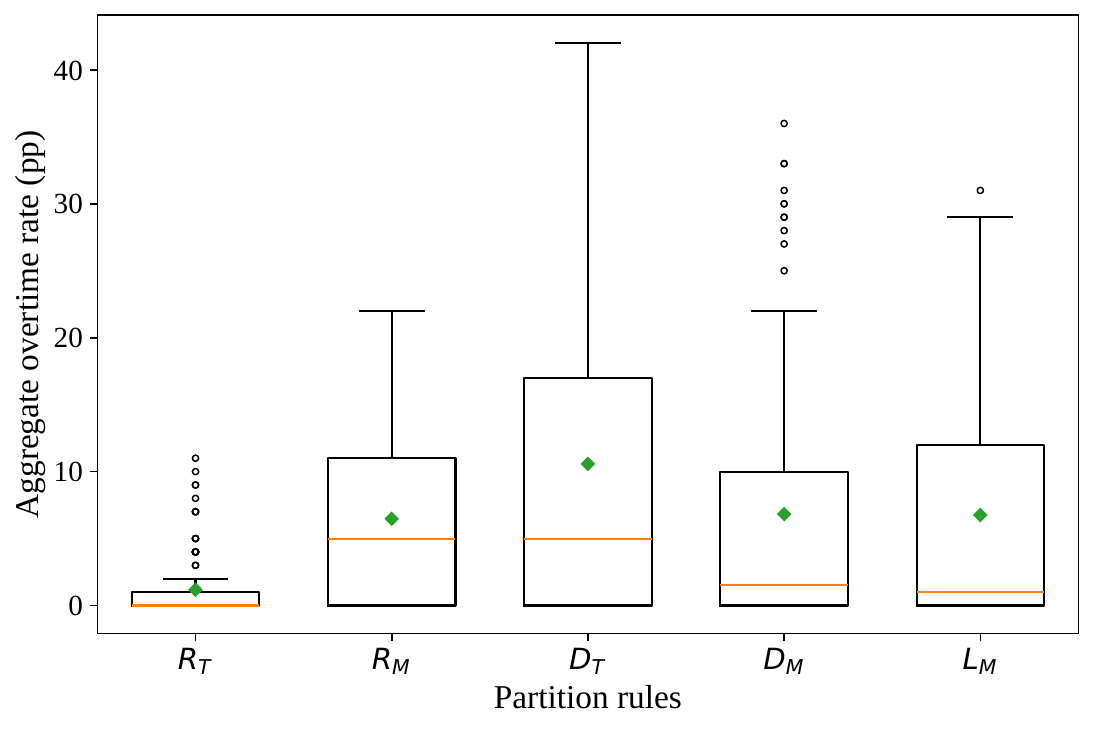}
        \label{fig:rateTE}
    }%
    \hfill
    \subfigure[Maximum overtime rate (pp)]{
        \includegraphics[width=0.48\textwidth]{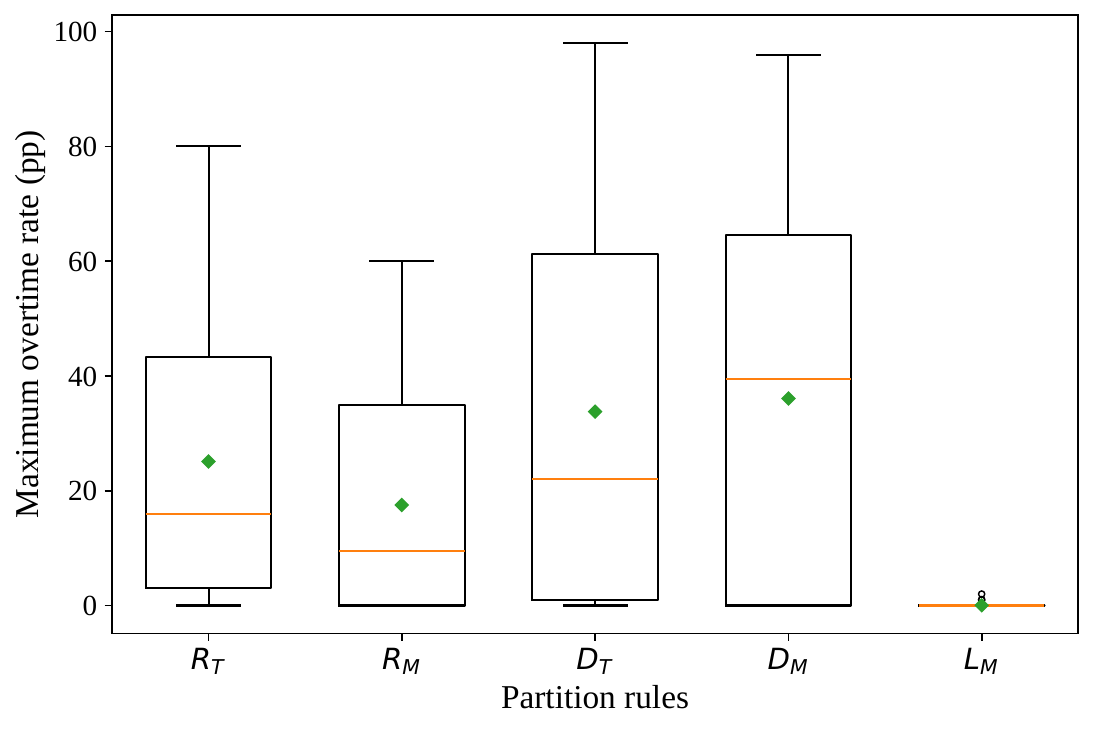}
        \label{fig:rateME}
    }%

    \caption{Out-of-sample performance of partition rules (relative gap to the best-performing rule in panels (a)--(d); absolute gap in panels (e)--(f)).}
    \label{fig:partition}
\end{figure}
\bibliographystyleapp{informs2014trsc}
\bibliographyapp{appendix.bib}

\end{document}